\documentclass[11pt]{article}
\usepackage[margin=1in]{geometry}
\usepackage[T1]{fontenc}
\IfFileExists{lmodern.sty}{\usepackage{lmodern}}{}
\usepackage{microtype}
\usepackage{xcolor}
\usepackage{amsmath,amssymb,amsthm}
\usepackage{enumitem}
\usepackage{esint}
\usepackage{hyperref}
\hypersetup{hidelinks}
\numberwithin{equation}{section}
\theoremstyle{plain}
\newtheorem{theorem}{Theorem}[section]
\newtheorem{proposition}[theorem]{Proposition}
\newtheorem{lemma}[theorem]{Lemma}
\newtheorem{corollary}[theorem]{Corollary}
\newtheorem*{maintheorem}{Theorem A}
\theoremstyle{definition}
\newtheorem{definition}[theorem]{Definition}
\theoremstyle{remark}
\newtheorem{remark}[theorem]{Remark}
\DeclareMathOperator{\Tr}{Tr}
\DeclareMathOperator{\ad}{ad}
\DeclareMathOperator{\vol}{vol}
\DeclareMathOperator{\Ric}{Ric}
\DeclareMathOperator{\Scal}{Scal}

\newcommand{\R}{\mathbb R}
\newcommand{\dd}{\,d}
\newcommand{\D}{\mathcal D}
\newcommand{\g}{\mathfrak g}

\title{Asymptotic Mean Value Laplacian\\ on equiregular sub-Riemannian manifolds}

\author{Fabrice Baudoin\thanks{F. Baudoin: Department of Mathematics, Aarhus University, Ny Munkegade 118, DK-8000 Aarhus C, Denmark. fbaudoin@math.au.dk} \and  Jonathan Junn\'e\thanks{J. Junn\'e: INRIA Rennes, Univ Rennes \& Institut de Recherche Mathématiques de Rennes,
CNRS UMR 6625 Rennes, Campus Beaulieu F-35042 Rennes Cedex, France. jonathan.junne@inria.fr} \and 
  David Tewodrose\thanks{D. Tewodrose: Department of Mathematics and Data Science, Vrije Universiteit Brussel, Pleinlaan 2, B-1050 Elsene, Belgium. david.tewodrose@vub.be}
}
\date{}
\begin{document}
\maketitle
\begin{abstract}
Let $(M,\D,g)$ be a smooth equiregular sub-Riemannian manifold equipped
with a smooth positive measure $\mu$. We study the small-scale limit of the metric-ball
mean-value operator
\[
  A_hf(x)=h^{-2}\fint_{B(x,h)}(f(q)-f(x))\,\dd\mu(q).
\]
Exact homogeneity yields the pointwise limit on Carnot groups. On a general
equiregular manifold, convergence for every smooth test function is
equivalent to convergence of the rescaled horizontal first moments of
metric balls in first-kind privileged coordinates. This criterion is
independent of $\mu$; when it holds, the principal symbol is determined by
the normalized second-moment tensor of the tangent unit ball, and the drift
satisfies an explicit change-of-measure formula. In step at most two, we
verify the criterion by combining a real-analytic finite-jet reduction with
tame integration, which rules out oscillation of the normalized moments. We
also compute the limit on Lie groups and prove unconditional distributional
convergence of the volume-weighted operators on every equiregular manifold.
\end{abstract}
\tableofcontents
\section{Introduction}\label{sec:introduction}
For a smooth function $f$ on $\R^n$, averaging over a small Euclidean ball centered at $x$ and performing a Taylor expansion yields that
\[
  \fint_{B(x,h)}f(q)\,\dd q-f(x)
  =\frac{h^2}{2(n+2)}\Delta f(x)+o(h^2) \qquad \text{as $h \downarrow 0$.}
\]
This identity motivates a metric-measure definition of the Laplace operator. If $d$ is the
Carnot--Carath\'eodory distance of a sub-Riemannian structure and $\mu$ is a
smooth positive measure, set
\begin{equation}\label{eq:AMV-def}
  A_hf(x):=\frac1{h^2}\fint_{B(x,h)}
  \bigl(f(q)-f(x)\bigr)\,\dd\mu(q).
\end{equation}
This paper aims to determine whether $A_hf$ has a limit as $h\downarrow0$
and, if so, to identify the limiting differential operator, known as the Asymptotic Mean Value (AMV) Laplacian.

In the Riemannian case, Taylor expansion in normal coordinates gives
\[
A_h f(x) = \sum_{i} \partial_{x_i} \tilde{f}(0) \left(\frac{1}{h}  \fint_{\mathbb{B}} \xi_i d \xi\right) + \frac{1}{2}\sum_{i,j} \partial_{x_i x_j}^2 \tilde{f}(0)  \fint_{\mathbb{B}} \xi_i \xi_j d \xi + o(1)
\]
  where $\mathbb{B}$ is the unit Euclidean ball and $\tilde{f} = f \circ \exp_x$. Central symmetry of $\mathbb{B}$ ensures that the first moments $ \fint_{\mathbb{B}} \xi_i d \xi$ vanish and that the second moments $\fint_{\mathbb{B}} \xi_i \xi_j d \xi$ are equal to a constant multiple of the Kronecker symbol, so that $A_h f(x)$ converges to the Laplace--Beltrami operator, up to a multiplicative constant.

In the general sub-Riemannian setting, the principal difficulty lies in the first moments. After rescaling in
privileged coordinates, metric balls converge to the unit ball of the
nilpotent approximation; their volumes and quadratic moments therefore
converge. Although the tangent ball is centrally symmetric, this yields only
an $o(1)$ estimate for the first moments of the rescaled balls, whereas the
expansion of \eqref{eq:AMV-def} divides these moments by $h$. Thus pointwise
convergence requires a sharper centering statement and does not follow from
tangent-ball convergence alone.

The present paper treats \eqref{eq:AMV-def} on a smooth equiregular
sub-Riemannian manifold $(M,\D,g)$ with an arbitrary smooth positive
measure. Its main conclusions are summarized as follows.
\begin{maintheorem}
Let $(M,\D,g)$ be a smooth equiregular sub-Riemannian manifold and let
$\mu$ be a smooth positive measure. If the structure has step at most two,
then $A_hf(x)$ converges for every $x\in M$ and
$f\in C_c^\infty(M)$ as $h \downarrow 0$. The convergence also holds in
$L^p_{\mathrm{loc}}(M,\mu)$ for every $1\le p<\infty$.
\end{maintheorem}
The pointwise assertion for structures of step at most two is
Theorem~\ref{thm:main}, and its local
$L^p$ refinement is Corollary~\ref{cor:step-two-consequences}. In the
Riemannian case the limit reduces,
for $\mu=e^W\,\mathrm{vol}_g$, to
\[
  A_hf\longrightarrow
  \frac1{2(n+2)}\Delta_gf
  +\frac1{n+2}\langle\nabla W,\nabla f\rangle_g.
\]
For $W=0$ this recovers the Riemannian computation outlined above. 
\subsection*{From Carnot groups to the first-moment criterion}
On a Carnot group the distance, Haar measure, and exponential coordinates
scale exactly. A direct Taylor expansion gives
\[
  A_hf\longrightarrow
  \frac1{2|\widehat B|}\sum_{i,j=1}^k
  \left(\int_{\widehat B}\xi_i\xi_j\,\dd\xi\right)X_iX_jf,
\]
where $X_1,\ldots,X_k$ is an orthonormal basis of the first layer and
$\widehat B$ is the unit Carnot--Carath\'eodory ball. Its horizontal
covariance matrix need not be scalar, so the limit need not be a multiple
of the intrinsic sub-Laplacian. Proposition
\ref{prop:Carnot-isotropic-covariance} gives a useful symmetry criterion for
scalarity.

On a general equiregular manifold, choose smoothly varying first-kind
privileged coordinates $\Phi_x$ and set
\[
  B_{x,h}:=\delta_{1/h}\bigl(\Phi_x^{-1}(B(x,h))\bigr).
\]
The sets $B_{x,h}$ converge to the unit ball $\widehat B_x$ of the
nilpotent approximation, which determines the normalized volume and all
second-order moments. The only remaining obstruction is the
horizontal first moment
\[
  \frac1h\int_{B_{x,h}}\xi_iJ_x(\delta_h\xi)\,\dd\xi,
  \qquad 1\le i\le k,
\]
where $J_x$ is the coordinate density of $\mu$ and
$k=\operatorname{rank}\D$. Theorem~\ref{thm:general-first-moment} shows
that convergence of these $k$ quantities is necessary and sufficient for
pointwise convergence of $A_h$ on all smooth test functions and identifies
the limit as
\[
  \mathcal A_\mu f(x)
  =\frac1{\widehat V_x}\left(
    \frac12\sum_{i,j=1}^kM_{ij}(x)X_iX_jf(x)
    +\sum_{i=1}^k\Gamma_i(x)X_if(x)
  \right).
\]
Here $\widehat V_x$ is the tangent volume of the unit ball $\widehat B_x$ in the nilpotent approximation, $(M_{ij})$ is the unnormalized horizontal second-moment matrix of
$\widehat B_x$, and $(\Gamma_i)$ records the first moments.
Equivalently, convergence is governed by the unweighted first moments of
the rescaled balls and is therefore independent of the chosen smooth
positive measure. Corollary~\ref{cor:smooth-measure-change} gives the
resulting change in drift when the measure is multiplied by a smooth
positive density.

The previous criterion applies directly to Lie groups. For a
connected Lie group with a left-invariant bracket-generating metric and
left Haar measure, first-kind coordinate balls are centrally symmetric. The only
drift comes from the linear term of the Haar density and is expressed by
the modular traces $\Tr(\ad_{X_j})$; see Theorem~\ref{thm:Lie-formula}.

The final part of Section~\ref{sec:general} is independent of the
first-moment criterion. We prove that the volume-weighted operators
$D_hA_h$ always converge in distributions; see Proposition \ref{prop:detailed-balance-convergence}. Consequently, every locally
bounded pointwise limit satisfies
\[
  \widehat V\,\mathcal A_\mu f
  =\frac12\operatorname{div}_\mu\left(
    \sum_{i,j=1}^kM_{ij}(X_jf)X_i
  \right).
\]
Thus, whenever the pointwise limit exists and is locally bounded and the
second-moment coefficients are locally Lipschitz, comparison with the
pointwise limit formula determines its drift almost everywhere from the
second-moment tensor and the measure. Theorem
\ref{thm:main} strengthens this distributional statement to convergence at
every point and local $L^p$ convergence.
\subsection*{The case of step at most two}
Section~\ref{sec:step-two} proves the first-moment condition using tame
geometry and o-minimality.  The argument begins with the uniform estimate
\[
  \widehat B_x(1-Ch)\subset B_{x,h}\subset\widehat B_x(1+Ch),
\]
which bounds the horizontal first moments by $O(h)$. At a fixed base
point, a sufficiently high weighted Taylor jet then produces a
real-analytic two-step comparison structure. Subanalyticity of its
distance and tame integration make the comparison moments constructible;
one-variable o-minimality then excludes oscillation after division by $h$.  The same uniform centering estimate yields local boundedness and local
$L^p$ convergence. We point out that the use of tame geometry and o-minimality ideas
in sub-Riemannian geometry is not new; see \cite{BonnardTrelat2001} and particularly the
recent work \cite{LeDonneLerarioNalonPaddeuRizzi2026}.
\subsection*{Literature}
Asymptotic mean-value operators on metric measure spaces have been studied
in \cite{MinneTewodrose2020,AKS1,AKS2}. The symmetrized AMV operator and
its relation to the non-symmetrized one were studied in
\cite{MinneTewodrose2025}; see also \cite{DiasTewodrose2026} for the
spectral theory of the symmetrized family. In the sub-Riemannian setting, asymptotic
mean-value formulas on a Carnot group have been obtained for balls of a
homogeneous \emph{gauge}: \cite{FerrariLiuManfredi2014} treats
Kor\'anyi gauge balls on Heisenberg groups,
\cite{FerrariPinamonti2015} treats general Carnot groups, and
\cite[Theorem~1.2]{AKS2} covers every homogeneous pseudonorm of the form
$\rho(z)=F(\lVert z^{(1)}\rVert,z')$, with $F$ of class $C^1$ and
$\partial_1F>0$ away from the origin, $z^{(1)}$ denoting the horizontal component;
\cite[Theorem~1.3]{AKS2} treats Kor\'anyi--Reimann gauge balls on step-two
Carnot groups. For genuine Carnot--Carath\'eodory balls, the case of the first Heisenberg
group with Lebesgue measure is \cite[Proposition~2.1]{MinneTewodrose2020}.
Theorem~\ref{thm:Carnot-limit} below recovers this Heisenberg result. Sub-Laplacians obtained from Hamiltonian or intrinsic geodesic random walks are investigated in
\cite{GordinaLaetsch2017,GordinaLaetsch2016,BoscainNeelRizzi2017}; those constructions average along geodesic data rather than over metric balls.

\section{Pointwise limits, the first-moment criterion and distributional convergence}
\label{sec:general}
We begin with Carnot groups, the tangent cones to general equiregular
sub-Riemannian manifolds; there the pointwise limit easily follows from exact
scaling. We then isolate the sole obstruction on an
equiregular manifold for the existence of pointwise limits, apply the resulting criterion to Lie groups, and
conclude with an unconditional distributional convergence result.
\subsection{Carnot groups}
\label{subsec:Carnot-groups}
Let $G$ be a Carnot group with identity element $e$ and stratified Lie algebra
\[
  \g=V_1\oplus\cdots\oplus V_s,
  \qquad [V_1,V_j]=V_{j+1}\quad(1\le j<s),
\]
and let $d$ be the Carnot-Carath\'eodory metric on $G$ induced by an inner product $\langle \cdot, \cdot \rangle$ on $V_1$.
Let $X_1,\ldots,X_k$ be an orthonormal basis of $V_1$ and complete it
to a graded basis $X_1,\ldots,X_n$ of $\g$. Write $w_a=j$ when
$X_a\in V_j$. We use the associated first-kind exponential coordinates
\[
(\xi_1,\ldots,\xi_n) \in \mathbb{R}^n \quad \leftrightarrow \quad \exp_G\left( \sum_{a=1}^n \xi_a X_a\right) \in G
\]where $\exp_G : \g \to G$ is the usual exponential map, and normalize Haar measure $\mu$ to be Lebesgue measure in these coordinates. If $\widehat B$ is the unit ball, set
\begin{equation*}
  \widehat V=|\widehat B|,
  \qquad
  M_{ij}=\int_{\widehat B}\xi_i\xi_j\,\dd\xi.
\end{equation*}
The group dilations $(\delta_h)_{h>0}$ satisfy
\[
  B(e,h)=\delta_h\widehat B
\]
where $B(e,h)$ is the ball with center $e$ and radius $h$.
\begin{theorem}[Carnot-group limit]
\label{thm:Carnot-limit}
For every $f\in C_c^\infty(G)$, locally uniformly in $x$,
\begin{equation}\label{eq:Carnot-limit}
  A_hf(x)\longrightarrow
  \frac1{2\widehat V}
  \sum_{i,j=1}^kM_{ij}X_iX_jf(x).
\end{equation}
\end{theorem}
\begin{proof}
Left invariance, the dilation identity, and the homogeneity of Haar measure
give
\[
  A_hf(x)=\frac1{h^2\widehat V}\int_{\widehat B}
  \left[f\!\left(x\exp_G\!\left(\sum_{a=1}^n
  h^{w_a}\xi_aX_a\right)\right)-f(x)\right]\dd\xi.
\]
For $Y_{h,\xi}:=\sum_a h^{w_a}\xi_aX_a$, Taylor's formula for the
time-one flow gives
\[
  f(x\exp_GY_{h,\xi})-f(x)
  =Y_{h,\xi}f(x)+\frac12Y_{h,\xi}^2f(x)+O(h^3).
\]
The remainder is uniform for $(x,\xi)\in G\times\widehat B$.  Indeed,
$\overline{\widehat B}$ is compact, all left-invariant derivatives of
$f\in C_c^\infty(G)$ of order at most three are bounded on $G$, and
$\sup_{\xi\in\widehat B}\lVert Y_{h,\xi}\rVert\le Ch$ for $0<h\le1$ in
any fixed norm on $\mathfrak g$.  Expanding the first two terms and
retaining the terms of weighted degree at most two therefore yields, with
the same uniformity,
\[
\begin{split}
  f(x\exp_GY_{h,\xi})-f(x)
  &=h\sum_{i=1}^k\xi_iX_if(x)
    +h^2\sum_{w_a=2}\xi_aX_af(x)\\
  &\quad+\frac{h^2}{2}\sum_{i,j=1}^k
       \xi_i\xi_jX_iX_jf(x)+O(h^3).
\end{split}
\]
In first-kind coordinates, group inversion is $\xi\mapsto-\xi$.
Consequently, $\widehat B=-\widehat B$, and both linear terms integrate to
zero. Dividing by
$h^2\widehat V$ proves \eqref{eq:Carnot-limit}.
\end{proof}
The covariance tensor
$\bigl(\int_{\widehat B}\xi_i\xi_j\,\dd\xi\bigr)_{ij}$ need not be
scalar, so the limit need not be a multiple of the sub-Laplacian. The
following proposition gives a sufficient symmetry condition. Write
\[
  \xi_H:=\sum_{i=1}^k\xi_iX_i,\qquad
  K_G:=\{A\in\mathrm O(V_1):A\text{ extends to a graded automorphism of }
  \mathfrak g\}.
\]
\begin{proposition}[Symmetry criterion for scalar covariance]
\label{prop:Carnot-isotropic-covariance}
The covariance operator $\mathbf M:V_1\to V_1$ defined by
\begin{equation}\label{eq:Carnot-covariance-form}
  \langle\mathbf M u,v\rangle
  :=\int_{\widehat B}
  \langle\xi_H,u\rangle\langle\xi_H,v\rangle\,\dd\xi
\end{equation}
commutes with $K_G$. Moreover, if the representation of $K_G$ on $V_1$ is
irreducible, then $$\mathbf M=m_G I,\qquad m_G = \frac{1}{k} \int_{\widehat B} |\xi_H|^2 \dd\xi, $$ and therefore
\[
  A_hf\longrightarrow\kappa_G\Delta_Hf,\qquad
  \kappa_G=\frac{m_G}{2\widehat V}
  ,\qquad
  \Delta_H=\sum_{i=1}^kX_i^2.
\]
This happens, in particular, when $K_G$ acts transitively on the unit
sphere.
\end{proposition}
\begin{proof}
Since $\mathfrak g$ is generated by $V_1$, a graded automorphism is
determined by its restriction to $V_1$; hence every $A\in K_G$ admits a
unique graded extension $\phi_A$, and $A\mapsto\phi_A$ is an injective
continuous group homomorphism. Extendability is a closed condition on
$A$, so $K_G$ is a compact subgroup of $\mathrm O(V_1)$. Consequently,
$A\mapsto|\det\phi_A|$ is a continuous homomorphism from a compact group
into $(\R_{>0},\cdot)$ and is therefore identically one. Let $F_A$ be the
automorphism of $G$ integrating $\phi_A$. In exponential coordinates,
$F_A=\phi_A$ is linear and preserves Lebesgue measure, and orthogonality
on $V_1$ makes $F_A$ a Carnot--Carath\'eodory isometry fixing the identity,
so $\phi_A\widehat B=\widehat B$. Since $\phi_A$ is graded,
$(\phi_A\eta)_H=A\eta_H$ for any $\eta \in \g$. Substituting $\xi=\phi_A\eta$ in
\eqref{eq:Carnot-covariance-form} gives
\[
  \langle\mathbf Mu,v\rangle
  =\int_{\widehat B}\langle A\eta_H,u\rangle\langle A\eta_H,v\rangle\dd\eta
  =\langle\mathbf MA^{\top}u,A^{\top}v\rangle
  =\langle A\mathbf MA^{\top}u,v\rangle.
\]
Thus $\mathbf M=A\mathbf MA^{\top}$, or equivalently
$A\mathbf M=\mathbf M A$. Since $\mathbf M$ is symmetric and positive
definite, its eigenspaces are $K_G$-invariant. Irreducibility therefore
forces $\mathbf M=m_GI$. Finally, a transitive orthogonal action is
irreducible.
\end{proof}
\begin{remark}[Basic scalar-covariance examples]
For $\R^k$, $K_G=\mathrm O(k)$ and
$M_{ij}=|B_{\R^k}(0,1)|\delta_{ij}/(k+2)$, giving
$\kappa_G=1/(2(k+2))$. For the standard Heisenberg group $\mathbb H^m$, an
element $A\in\mathrm O(2m)$ extends to a graded automorphism if and only if
$A^{*}\omega=\lambda\omega$ for some $\lambda\in\R$, where $\omega$ is the
symplectic form on $V_1$ defined by the bracket. Hence
$\mathrm U(m)=\mathrm O(2m)\cap\mathrm{Sp}(2m,\R)$ is contained in $K_G$
and acts transitively on the unit sphere of $V_1\simeq\R^{2m}$. See
\cite{LeDonneOttazzi2016} for the structure of the isometry group of a
Carnot group. For the free
Carnot group of any step, every element of $\mathrm O(k)$ extends by the
universal property. Thus the covariance is scalar in all three families.
\end{remark}
\begin{remark}[A product group with non-scalar covariance]
\label{rem:non-scalar-covariance}
 Let $G_0$ be a Carnot group of homogeneous dimension $Q_0$
whose own covariance is scalar, write $\widehat B_0$ for its unit
Carnot--Carath\'eodory ball and
$m_0:=\int_{\widehat B_0}\xi_1^2\,\dd\xi$, and let
$G=\R\times G_0$ carry the product metric, the extra horizontal generator
being denoted $X_0$. Then
$\widehat d\bigl((t,p),(0,e)\bigr)^2=t^2+\widehat d_0(e,p)^2$, so the slice of
$\widehat B$ at height $t$ is $\widehat B_0\bigl(\sqrt{1-t^2}\bigr)$, and
homogeneity of the dilations of $G_0$ gives
\[
  M_{00}=|\widehat B_0|\int_{-1}^1t^2(1-t^2)^{Q_0/2}\,\dd t,
  \qquad
  M_{11}=m_0\int_{-1}^1(1-t^2)^{(Q_0+2)/2}\,\dd t,
\]
all mixed moments vanishing by symmetry. Both integrals are Beta
integrals, and, with $p=Q_0/2$,
\[
 \frac{B(\tfrac32,p+1)}{B(\tfrac12,p+2)}=\frac1{2p+2},
\]
so
\begin{equation}\label{eq:product-covariance-ratio}
  \frac{M_{00}}{M_{11}}
  =\frac1{Q_0+2}\cdot\frac{|\widehat B_0|}{m_0}.
\end{equation}
Hence the covariance of $\R\times G_0$ is scalar if and only if
\begin{equation}\label{eq:scalar-product-criterion}
  \frac{m_0}{|\widehat B_0|}=\frac1{Q_0+2},
\end{equation}
that is, if and only if the unit ball of $G_0$ has the normalized
horizontal second moment of a Euclidean ball of dimension $Q_0$. For
$G_0=\R^k$ this holds, as expected. For \(G_0=\mathbb H^1\), the geodesic cylindrical coordinates with explicit Jacobian \cite[(36), (45)]{BarilariBeschastnyiLerario2020} give, after
integration in the radial and angular variables,
\[
 \frac{m_0}{|\widehat B_0|}
 =\frac{2}{3}\,
 \frac{\displaystyle\int_0^{2\pi}
   \frac{(1-\cos w)\bigl(2-2\cos w-w\sin w\bigr)}{w^6}\,\dd w}
 {\displaystyle\int_0^{2\pi}
   \frac{2-2\cos w-w\sin w}{w^4}\,\dd w}.
\]
Writing
\(\operatorname{Si}(s):=\int_0^s(\sin u)/u\,\dd u\), repeated integration by parts
therefore yields
\[
 \frac{m_0}{|\widehat B_0|}
 =
 \frac{3-2\pi\operatorname{Si}(2\pi)
          +16\pi\operatorname{Si}(4\pi)}
 {30\bigl(1+2\pi\operatorname{Si}(2\pi)\bigr)}
 =0.232392\ldots>\frac16=\frac1{Q_0+2}.
\]
Thus the covariance of \(\mathbb R\times\mathbb H^1\) is not scalar.
\end{remark}
\subsection{Equiregular manifolds and tangent balls}
\label{subsec:general-definitions}
Let $M$ be a connected smooth manifold of dimension $n$. A sub-Riemannian
structure on $M$ is a pair $(\D,g)$, where $\D\subset TM$ is a
bracket-generating vector subbundle, and $g$ is a smooth inner product on
$\D$. Its flag is
\begin{equation}\label{eq:distribution-flag}
  \D^0=\{0\},\qquad \D^1=\D,\qquad
  \D^{j+1}=\D^j+[\D^j,\D].
\end{equation}
We assume that the structure is equiregular, so the integers
$n_j:=\dim\D_x^j$ are independent of $x$. If $s$ is the step, then
\begin{equation}\label{eq:growth-data}
  0=n_0<n_1=k<\cdots<n_s=n,
  \qquad
  Q:=\sum_{j=1}^s j(n_j-n_{j-1}).
\end{equation}
Here $k=\operatorname{rank}\D$ and $Q$ is the homogeneous dimension. For
$n_{j-1}<a\leq n_j$, set $w_a:=j$.
The Carnot--Carath\'eodory distance $d$ is the infimum of the $g$-lengths of
absolutely continuous curves tangent to $\D$ almost everywhere. By the
Chow--Rashevskii theorem, $d$ is finite and induces the manifold topology.
General references for sub-Riemannian geometry are
\cite{Montgomery2002,ABB2020}.

Let $O\subset M$ be open. An \emph{adapted bracket frame} on $O$ is a frame
$X_1,\ldots,X_n$ with the following properties: (1) the fields
$X_1,\ldots,X_k$ are an orthonormal frame of $\D$, (2)
$X_1,\ldots,X_{n_j}$ span $\D^j$ for any $j$, (3) if $n_{j-1}<a\leq n_j$,  the class of
$X_a$ in $\D^j/\D^{j-1}$ is represented by a bracket of length $j$ in the
horizontal fields. Such a frame always exists locally. Let us recall the notion of privileged coordinates, see \cite[Sections~4--5]{Bellaiche1996}, \cite[Section~2.1]{Jean2014}, or
\cite[Chapter~10]{ABB2020} for further details.

\begin{definition}[Privileged coordinates]
A coordinate system $z=(z_1,\ldots,z_n)$ centered at $x$ is
\emph{privileged} if $z_a$ has nonholonomic order $w_a$ at $x$ for every
$a$. Equivalently, $w_a$ is the least length of an iterated horizontal
derivative of $z_a$ that does not vanish at $x$.
\end{definition}

We use canonical privileged coordinates of the first kind. If $Y$ is a smooth vector
field, $\exp(Y)(x)$ denotes its time-one flow from $x$.
Given an adapted bracket frame $X_1,\ldots,X_n$ on $O$ and  $x\in O$, set
\begin{equation}\label{eq:first-kind-chart}
  \Phi_x(z):=\exp\!\left(\sum_{a=1}^n z_aX_a\right)(x),
  \qquad z\in\R^n.
\end{equation}
Then the following holds.
\begin{proposition}[First-kind charts]
\label{prop:first-kind-privileged}
Let $(M,\D,g)$ be equiregular and let $X_1,\ldots,X_n$ be an adapted bracket
frame on $O$. After shrinking $O$, each $\Phi_x^{-1}$ is a privileged
coordinate system centered at $x \in O$. Moreover, for every $K\Subset O$, the maps
$\{\Phi_x\}_{x\in K}$ are diffeomorphisms on a common neighborhood of
$0\in\R^n$ and depend smoothly on $(x,z)$.
\end{proposition}
This is a standard construction; see
\cite{Bellaiche1996,Jean2014,ABB2020}. We henceforth fix an adapted bracket
frame on $O$ and use the resulting charts. Uniform assertions refer to
centers in a compact set $K\Subset O$.
For sufficiently small $r_0>0$, every horizontal curve of length less than
$r_0$ starting in $K$ remains in $O$. Thus the ambient distance and the
distance obtained by restricting curves to $O$ give the same small balls.
At $x\in O$, the nilpotentization is the graded Lie algebra
\begin{equation}\label{eq:nilpotent-algebra-general}
  \operatorname{gr}_x(\D)
  :=\bigoplus_{j=1}^s \D_x^j/\D_x^{j-1},
\end{equation}
whose bracket is induced by brackets of local sections of the flag. Let
$G_x$ be the corresponding Carnot group, endowed with the
left-invariant sub-Riemannian metric for which the classes of
$X_1(x),\ldots,X_k(x)$ are orthonormal in the first layer. Exponential
coordinates of the first kind give the identifications
\begin{equation}\label{eq:tangent-coordinate-identification}
  G_x\simeq\R^n\simeq
  \bigoplus_{j=1}^s\R^{n_j-n_{j-1}}.
\end{equation}
We denote the Carnot--Carath\'eodory distance on $G_x$ by $\widehat d_x$ and set
\begin{equation}\label{eq:tangent-ball-general}
  \widehat B_x(r):=\{\xi\in\R^n:\widehat d_x(0,\xi)<r\},
  \qquad \widehat B_x:=\widehat B_x(1).
\end{equation}
The anisotropic dilations are
\begin{equation}\label{eq:anisotropic-dilations}
  \delta_h\xi=(h^{w_1}\xi_1,\ldots,h^{w_n}\xi_n),
  \qquad h>0,
\end{equation}
with Jacobian $h^Q$, and
\begin{equation}\label{eq:tangent-distance-homogeneous}
  \widehat d_x(\delta_\lambda\xi,\delta_\lambda\eta)
  =\lambda\widehat d_x(\xi,\eta).
\end{equation}
For $h>0$, define the rescaled distance and rescaled unit ball by
\begin{align}
  d_{x,h}(\xi,\eta)
  &:=h^{-1}d\bigl(\Phi_x(\delta_h\xi),\Phi_x(\delta_h\eta)\bigr),
  \label{eq:rescaled-distance-general}\\
  B_{x,h}
  &:=\delta_{1/h}\bigl(\Phi_x^{-1}(B(x,h))\bigr)
    =\{\xi:d_{x,h}(0,\xi)<1\}.
  \label{eq:rescaled-ball-general}
\end{align}
On every bounded coordinate set, the nilpotent approximation theorem gives
\begin{equation}\label{eq:distance-convergence-pointwise}
  d_{x,h}\longrightarrow\widehat d_x
  \qquad\text{uniformly as }h\downarrow0.
\end{equation}
The convergence is uniform in $x$ on compact subsets of $O$. The rescaled
horizontal fields
\begin{equation}\label{eq:general-rescaled-fields}
  X_i^{x,h}:=h(\delta_{1/h})_*(\Phi_x^{-1})_*X_i,
  \qquad i=1,\ldots,k,
\end{equation}
converge in $C^\infty$ on compact sets to the model nilpotent fields. Write
$\widehat X_x=(\widehat X_{1,x},\ldots,\widehat X_{k,x})$ for this frame. We
use the following standard results; see
\cite{Mitchell1985,Bellaiche1996,Jean2014,DonMagnani2023}.
\begin{proposition}[Uniform ball-box estimate and tangent-ball trapping]
\label{prop:uniform-ball-control}
For every compact $K\Subset O$, there exist $C\ge1$ and $r_0>0$ such that
\begin{equation}\label{eq:uniform-ball-box}
  \delta_r([-C^{-1},C^{-1}]^n)
  \subset \Phi_x^{-1}(B(x,r))
  \subset \delta_r([-C,C]^n)
\end{equation}
for $x\in K$ and $0<r<r_0$. There is also a function
$\varepsilon_K(h)\downarrow0$ such that
\begin{equation}\label{eq:trapping-general}
  \widehat B_x(1-\varepsilon_K(h))
  \subset B_{x,h}
  \subset \widehat B_x(1+\varepsilon_K(h))
\end{equation}
for every $x\in K$ and all sufficiently small $h$.
\end{proposition}
\begin{proof}
The first assertion is the compact-uniform ball-box theorem. It places all
$B_{x,h}$, with $x\in K$ and $h$ small, in a common compact coordinate box.
On a larger box, \eqref{eq:distance-convergence-pointwise} is uniform in
$x$. Thus
$|d_{x,h}(0,\xi)-\widehat d_x(0,\xi)|\leq\varepsilon_K(h)$ there.
This gives \eqref{eq:trapping-general}.
\end{proof}
\begin{remark}
The proof of the ball-box theorem gives the estimate $\varepsilon_K(h)=O(h^{1/s})$. We
will prove in Proposition~\ref{prop:boundedness} that for $s\le2$ this improves
to $\varepsilon_K(h)=O(h)$.
\end{remark}
\subsection{First-moment criterion}
\label{subsec:first-moment-condition}
Let $\mu$ be a smooth positive measure on $M$. In the first-kind chart $\Phi_x^{-1}$ on $O \subset M$
centered at $x$, write
\begin{equation}\label{eq:J-density-general}
  \Phi_x^*\mu=J_x(z)\,\dd z,
  \qquad J_x(0)>0.
\end{equation}
The induced tangent measure on $G_x\simeq\R^n$ is
\begin{equation}\label{eq:tangent-measure-general}
  \widehat\mu_x:=J_x(0)\,\dd\xi.
\end{equation}
Set
\begin{equation}\label{eq:Vx}
  \widehat V_x:=\widehat\mu_x(\widehat B_x)
  =J_x(0)|\widehat B_x|,
\end{equation}
where $|\cdot|$ denotes Lebesgue measure. For $1\leq i,j\leq k$, set
\begin{equation}\label{eq:Mij-general}
  M_{ij}(x):=J_x(0)\int_{\widehat B_x}\xi_i\xi_j\,\dd\xi.
\end{equation}
The normalized covariance
\begin{equation}\label{eq:normalized-covariance}
  C_{ij}(x):=\frac{M_{ij}(x)}{\widehat V_x}
\end{equation}
is independent of the density factor $J_x(0)$. The next lemma shows that
$C$ is an intrinsic, continuous, positive-definite section of
$\operatorname{Sym}^2\D$.
\begin{lemma}[Intrinsic character of $\widehat V$ and $M$]
\label{lem:intrinsic-moments}
The scalar $\widehat V_x$ is independent of the adapted bracket frame used
to build $\Phi_x$, and the matrices $(M_{ij}(x))$ obtained from different
adapted frames are the components of a well-defined positive-definite section of
$\operatorname{Sym}^2\D$ over $O$. More precisely,
\begin{equation}\label{eq:Vx-intrinsic}
  \widehat V_x=\lim_{h\downarrow0}h^{-Q}\mu\bigl(B(x,h)\bigr),
\end{equation}
and $M$ transforms by orthogonal conjugation under an orthogonal change of
horizontal frame. Moreover, $x\mapsto\widehat V_x$ and, in every smooth
adapted frame, $x\mapsto M_{ij}(x)$ are continuous and locally bounded, and
$\widehat V$ is locally bounded away from zero.
\end{lemma}
\begin{proof}
Fix $x\in O$. For $f\in C^\infty(M)$, Taylor expansion in the first-kind
chart $\Phi_x^{-1}$ gives, uniformly for $\xi$ ranging over bounded subsets of $\mathbb{R}^n$,
\[
  \frac{f(\Phi_x(\delta_h\xi))-f(x)}h
  \longrightarrow \sum_{i=1}^k\xi_iX_if(x).
\]
The trapping estimate and homogeneity give
$|B_{x,h}\mathbin\triangle\widehat B_x|\to0$. After the change of variables
$z=\delta_h\xi$, it follows that
\[
  h^{-Q}\mu(B(x,h))
  =\int_{B_{x,h}}J_x(\delta_h\xi)\,\dd\xi
  \longrightarrow J_x(0)|\widehat B_x|=\widehat V_x,
\]
which proves \eqref{eq:Vx-intrinsic}. The same argument, applied to two
smooth functions $f,g$, yields
\begin{equation}\label{eq:M-intrinsic-limit}
\begin{split}
  &\lim_{h\downarrow0}h^{-Q-2}\int_{B(x,h)}
    (f(q)-f(x))(g(q)-g(x))\,\dd\mu(q)\\
  &\qquad=
    \sum_{i,j=1}^kM_{ij}(x)X_if(x)X_jg(x).
\end{split}
\end{equation}
The left-hand side is intrinsic, and the right-hand side depends only on
$\dd f|_{\D_x}$ and $\dd g|_{\D_x}$. Since these horizontal differentials
may be prescribed arbitrarily, \eqref{eq:M-intrinsic-limit} defines an
intrinsic element of $\operatorname{Sym}^2\D_x$. Thus, if
$X_a'=\sum_iR_{ia}X_i$ with $R\in\mathrm O(k)$, then
\[
  M'=R^\top MR.
\]
The tensor is positive definite because the integral of the square of any
nonzero horizontal linear functional over the open set $\widehat B_x$ is
strictly positive.
For continuity, fix $K\Subset O$ and choose a compact coordinate box
$\mathcal C$ containing $\widehat B_x(2)$ for every $x\in K$. Each map
$(x,\xi)\mapsto d_{x,h}(0,\xi)$ is continuous on $K \times \mathcal{C}$, and by
\eqref{eq:distance-convergence-pointwise} these maps converge uniformly to $(x,\xi)\mapsto\widehat d_x(0,\xi)$, which is
therefore continuous. Set
\[
  \varepsilon(x,y):=\sup_{\xi\in\mathcal C}
  |\widehat d_x(0,\xi)-\widehat d_y(0,\xi)|.
\]
Then $\varepsilon(x,y)\to0$ as $y\to x$, and for
$\varepsilon:=\varepsilon(x,y)<1$,
\[
  \widehat B_x(1-\varepsilon)\subset\widehat B_y
  \subset\widehat B_x(1+\varepsilon).
\]
Thus, by homogeneity,
\[
  |\widehat B_x\mathbin\triangle\widehat B_y|
  \le\bigl((1+\varepsilon)^Q-(1-\varepsilon)^Q\bigr)|\widehat B_x|
  \longrightarrow0.
\]
The ball-box estimate bounds $|\widehat B_x|$ uniformly on $K$. Hence the
integrals of $1$ and $\xi_i\xi_j$ over $\widehat B_x$ depend continuously on
$x$. Since $J_x(0)$ is smooth, the definitions of $\widehat{V}_x$ and $M_{ij}(x)$, namely \eqref{eq:Vx} and
\eqref{eq:Mij-general}, respectively, give the claimed continuity and local boundedness.
The positive continuous function $\widehat V$ is bounded away from zero on
compact subsets of $O$.
\end{proof}
\begin{proposition}[Tangent-ball moment convergence]
\label{prop:tangent-ball-convergence-general}
Let $K\Subset O$. For $x\in K$ and small $h$, the sets $B_{x,h}$ lie in a
common compact subset of $\R^n$. Moreover,
\begin{equation}\label{eq:indicator-convergence-general}
  \sup_{x\in K}|B_{x,h}\mathbin\triangle\widehat B_x|
  \longrightarrow0.
\end{equation}
Consequently, locally uniformly in $x\in O$,
\begin{align}
  \int_{B_{x,h}}\xi_i\xi_jJ_x(\delta_h\xi)\,\dd\xi
  &\longrightarrow M_{ij}(x),
  &&1\leq i,j\leq k,
  \label{eq:second-moment-convergence-general}\\
  \int_{B_{x,h}}J_x(\delta_h\xi)\,\dd\xi
  &\longrightarrow\widehat V_x.
  \label{eq:volume-convergence-general}
\end{align}
More generally, if $\Psi(x,\xi)$ is continuous on a neighborhood of a common
compact set containing the tangent and rescaled balls, then, for
$\theta\in\{0,1\}$,
\begin{equation}\label{eq:general-moment-convergence}
 \int_{B_{x,h}}\Psi(x,\xi)J_x(\delta_h\xi)^\theta\,\dd\xi
 \longrightarrow
 J_x(0)^\theta\int_{\widehat B_x}\Psi(x,\xi)\,\dd\xi .
\end{equation}
The convergence is locally uniform in $x$.
\end{proposition}
\begin{proof}
Common compact containment follows from the uniform ball-box estimate \eqref{eq:uniform-ball-box}. By
\eqref{eq:trapping-general} and homogeneity,
\[
  \sup_{x\in K}|B_{x,h}\mathbin\triangle\widehat B_x|
  \le \sup_{x\in K}|\widehat B_x|
  \bigl((1+\varepsilon_K(h))^Q-(1-\varepsilon_K(h))^Q\bigr)
  \longrightarrow0.
\]
The supremum is finite by the uniform ball-box estimate. On the common
compact set, $J_x(\delta_h\xi)\to J_x(0)$ uniformly. The same
symmetric-difference estimate and uniform continuity prove
\eqref{eq:general-moment-convergence}; the two special cases \eqref{eq:second-moment-convergence-general} and \eqref{eq:volume-convergence-general}
follow immediately.
\end{proof}
In first-kind exponential coordinates, inversion in $G_x$ is
$\xi\mapsto-\xi$. Symmetry and left invariance of $\widehat d_x$ therefore give
$\widehat d_x(0,\xi)=\widehat d_x(0,-\xi)$, so
$\widehat B_x=-\widehat B_x$ and
\begin{equation}\label{eq:tangent-first-moments-zero}
  \int_{\widehat B_x}\xi_a\,\dd\xi=0,
  \qquad a=1,\ldots,n.
\end{equation}
The Jacobian of $\delta_h$ gives
\begin{equation}\label{eq:Dh-definition-general}
  D_h(x):=\int_{B_{x,h}}J_x(\delta_h\xi)\,\dd\xi
  =h^{-Q}\mu(B(x,h)).
\end{equation}
With $A_h$ as in \eqref{eq:AMV-def}, we have therefore
\begin{equation}\label{eq:rescaled-operator-general}
  A_hf(x)=\frac1{h^2D_h(x)}
  \int_{B_{x,h}}
  \bigl(f(\Phi_x(\delta_h\xi))-f(x)\bigr)
  J_x(\delta_h\xi)\,\dd\xi.
\end{equation}
\begin{definition}[First-moment condition]
\label{def:first-moment-condition}
The first-kind coordinates at $x$ satisfy the \emph{first-moment condition}
with respect to $\mu$ if the limits
\begin{equation}\label{eq:Gamma-def-general}
  \Gamma_i(x):=\lim_{h\downarrow0}\frac1h
  \int_{B_{x,h}}\xi_iJ_x(\delta_h\xi)\,\dd\xi,
  \qquad i=1,\ldots,k,
\end{equation}
exist. We call the $\Gamma_i(x)$ the first-moment coefficients in the
chosen chart.
\end{definition}
Fix $x\in O$ and set
\begin{equation}\label{eq:beta-def-general}
  \beta_j(x):=(\partial_{z_j}\log J_x)(0),
  \qquad j=1,\ldots,k.
\end{equation}
Whenever they exist, denote the unweighted horizontal first moments by
\begin{equation}\label{eq:unweighted-first-moment-condition}
  \Lambda_i(x):=\lim_{h\downarrow0}\frac1h
  \int_{B_{x,h}}\xi_i\,\dd\xi.
\end{equation}
\begin{lemma}[First moments and the measure density]
\label{lem:unweighted-first-moments}
The first-moment condition holds at $x$ if and only if the limits
$\Lambda_i(x)$ exist for $i=1,\ldots,k$. In that case
\begin{equation}\label{eq:weighted-unweighted-first-moment-relation}
  \Gamma_i(x)=J_x(0)\Lambda_i(x)
  +\sum_{j=1}^kM_{ij}(x)\beta_j(x).
\end{equation}
\end{lemma}
\begin{proof}
On the common compact coordinate box,
\begin{equation}\label{eq:J-expansion-general}
  J_x(\delta_h\xi)
  =J_x(0)\left(
    1+h\sum_{j=1}^k\beta_j(x)\xi_j+O(h^2)
  \right).
\end{equation}
Only horizontal variables occur at order $h$; all other coordinates have
weight at least two. Multiplying by $\xi_i$, integrating, and dividing by
$h$ gives
\begin{align*}
  \frac1h\int_{B_{x,h}}\xi_iJ_x(\delta_h\xi)\,\dd\xi
  &=J_x(0)\frac1h\int_{B_{x,h}}\xi_i\,\dd\xi\\
  &\quad+J_x(0)\sum_{j=1}^k\beta_j(x)
    \int_{B_{x,h}}\xi_i\xi_j\,\dd\xi+o(1).
\end{align*}
Proposition~\ref{prop:tangent-ball-convergence-general}, applied with
$\Psi(x,\xi)=\xi_i\xi_j$ and $\theta=0$, shows that the quadratic integral
converges to $J_x(0)^{-1}M_{ij}(x)$. The resulting asymptotic identity proves both
the equivalence and the formula.
\end{proof}
Fix $x\in O$. A \emph{horizontal coordinate test family} at $x$ consists
of functions $f_1,\ldots,f_k\in C_c^\infty(M)$ such that
\begin{equation}\label{eq:coordinate-test-functions}
  f_i(\Phi_x(\xi))=\xi_i
\end{equation}
near $\xi=0$.
\begin{theorem}[First-moment criterion]
\label{thm:general-first-moment}
Fix $x\in O$ and a first-kind privileged chart $\Phi_x$. The following are
equivalent.
\begin{enumerate}[label=\textup{(\roman*)}, leftmargin=2em]
\item The unweighted limits $\Lambda_i(x)$ exist for $i=1,\ldots,k$.
\item The first-moment condition holds at $x$.
\item The limit $\lim_{h\downarrow0}A_hf(x)$ exists for every
  $f\in C_c^\infty(M)$.
\item The quantities $A_hf_i(x)$ converge for some horizontal coordinate
  test family.
\end{enumerate}
When these conditions hold,
\begin{equation}\label{eq:general-limit-convergence}
  A_hf(x)\longrightarrow\mathcal A_\mu f(x)
\end{equation}
for every $f\in C_c^\infty(M)$, where
\begin{equation}\label{eq:general-limit-operator}
  \mathcal A_\mu f(x)
  =\frac1{\widehat V_x}\left(
    \frac12\sum_{i,j=1}^kM_{ij}(x)X_iX_jf(x)
    +\sum_{i=1}^k\Gamma_i(x)X_if(x)
  \right).
\end{equation}
\end{theorem}
\begin{proof}
Lemma~\ref{lem:unweighted-first-moments} proves
\textup{(i)}$\Leftrightarrow$\textup{(ii)}. Assume \textup{(ii)}, fix a
compact box containing $B_{x,h}$ for all sufficiently small $h$, and for any $\xi \in B_{x,h}$ set
\begin{equation}\label{eq:general-Taylor-vector-field}
  Y_{h,\xi}:=\sum_{a=1}^n h^{w_a}\xi_aX_a.
\end{equation}
By the definition of $\Phi_x$,
$\Phi_x(\delta_h\xi)=\exp(Y_{h,\xi})(x)$. Taylor's formula for the time-one
flow gives, uniformly on this box,
\begin{equation}\label{eq:Taylor-general}
\begin{split}
  f(\Phi_x(\delta_h\xi))-f(x)
  &=h\sum_{i=1}^k\xi_iX_if(x)
    +h^2\sum_{w_a=2}\xi_aX_af(x)\\
  &\quad+\frac{h^2}{2}\sum_{i,j=1}^k
    \xi_i\xi_jX_iX_jf(x)+O(h^3).
\end{split}
\end{equation}
Indeed, $Y_{h,\xi}=O(h)$ in every fixed $C^m$ norm, and the displayed terms are
the components of weighted degree at most two in
$Y_{h,\xi}f+\frac12Y_{h,\xi}^2f$. Substitution into
\eqref{eq:rescaled-operator-general} gives
\begin{equation}\label{eq:expanded-general}
\begin{split}
  A_hf(x)=\frac1{D_h(x)}\Bigg[&
    \sum_{i=1}^kX_if(x)\frac1h
      \int_{B_{x,h}}\xi_iJ_x(\delta_h\xi)\,\dd\xi\\
  &+\sum_{w_a=2}X_af(x)
      \int_{B_{x,h}}\xi_aJ_x(\delta_h\xi)\,\dd\xi\\
  &+\frac12\sum_{i,j=1}^kX_iX_jf(x)
      \int_{B_{x,h}}\xi_i\xi_jJ_x(\delta_h\xi)\,\dd\xi
  \Bigg]+O(h).
\end{split}
\end{equation}
The first line converges by hypothesis, and the third one by \eqref{eq:second-moment-convergence-general} in
Proposition~\ref{prop:tangent-ball-convergence-general}. If $w_a=2$,
\eqref{eq:general-moment-convergence} in
Proposition~\ref{prop:tangent-ball-convergence-general}, with $\Psi(x,\xi)=\xi_a$ and
$\theta=1$, shows that the corresponding integral in the second line
converges to $J_x(0)\int_{\widehat B_x}\xi_a\,\dd\xi$ which is equal to $0$ by
\eqref{eq:tangent-first-moments-zero}. Since
$D_h(x)\to\widehat V_x>0$, formula
\eqref{eq:general-limit-operator} follows. Thus
\textup{(ii)}$\Rightarrow$\textup{(iii)}.
Implication \textup{(iii)}$\Rightarrow$\textup{(iv)} is obvious. Conversely, assume
\textup{(iv)}. For all sufficiently small $h$, the defining identity for $f_i$ holds on
$B(x,h)$, so
\begin{equation}\label{eq:coordinate-test-exact-identity}
  A_hf_i(x)=\frac1{hD_h(x)}
  \int_{B_{x,h}}\xi_iJ_x(\delta_h\xi)\,\dd\xi.
\end{equation}
Since $D_h(x)\to\widehat V_x>0$, convergence of the left-hand side is equivalent
to the existence of the limit $\Gamma_i(x)$. Thus
\textup{(iv)}$\Rightarrow$\textup{(ii)}, completing the equivalence.
\end{proof}
\begin{remark}
\label{rem:chart-independence}
Condition \textup{(iii)} of Theorem~\ref{thm:general-first-moment} refers to
no chart. Consequently, the first-moment condition holds at $x$ in one
first-kind privileged chart arising from a local adapted bracket frame if
and only if it holds in every such chart, and the limiting operator
$\mathcal A_\mu$ is independent of the chart used to compute it.
\end{remark}
Combining Theorem~\ref{thm:general-first-moment} with Lemma
\ref{lem:unweighted-first-moments} separates the metric and measure data: considering the normalized covariance $C_{ij}(x)=M_{ij}(x)/\widehat V_x$, we obtain
\begin{equation}\label{eq:general-limit-normalized}
\begin{split}
  \mathcal A_\mu f(x)
  &=\frac12\sum_{i,j=1}^kC_{ij}(x)X_iX_jf(x)\\
  &\quad+\sum_{i=1}^k\left(
    \frac{\Lambda_i(x)}{|\widehat B_x|}
    +\sum_{j=1}^kC_{ij}(x)\beta_j(x)
  \right)X_if(x).
\end{split}
\end{equation}
Thus the existence criterion and the quantities $C_{ij},\Lambda_i$ depend
only on the metric and the chosen chart; the smooth measure enters the
operator through the logarithmic density derivatives $\beta_j$.
Although $\Lambda_i$, $\Gamma_i$, and $\beta_i$ are chart-dependent, the
tensor $(C_{ij})$ and, whenever the first moments converge, the full
operator in \eqref{eq:general-limit-normalized} are intrinsic.
\begin{corollary}[Smooth changes of measure]
\label{cor:smooth-measure-change}
Let $\mu_W=e^W\mu$ with $W\in C^\infty(M)$. Pointwise convergence on all
smooth test functions with respect to $\mu_W$ holds at $x$ if and only if
it holds with respect to $\mu$. When the two limiting operators exist,
\begin{equation}\label{eq:smooth-measure-change}
  \mathcal A_{\mu_W}f(x)
  =\mathcal A_\mu f(x)
  +\sum_{i,j=1}^kC_{ij}(x)(X_jW)(x)X_if(x).
\end{equation}
\end{corollary}
\begin{proof}
The metric balls and the unweighted moments $\Lambda_i$ do not depend on
the measure. Lemma~\ref{lem:unweighted-first-moments} and Theorem
\ref{thm:general-first-moment} therefore give the equivalence. In the
chosen chart, the density of $\mu_W$ is $e^{W\circ\Phi_x}J_x$, so its
horizontal logarithmic derivatives are $\beta_j+X_jW$; indeed
$\partial_{z_j}(W\circ\Phi_x)(0)
=\frac{\dd}{\dd t}\big|_{t=0}W\bigl(\exp(tX_j)(x)\bigr)=X_jW(x)$ for
$j\le k$. Substitution in \eqref{eq:general-limit-normalized} proves
\eqref{eq:smooth-measure-change}.
\end{proof}
\subsection{Lie groups}
\label{subsec:Lie-groups}
The first-moment criterion applies directly to left-invariant
structures. Let $G$ be a connected Lie group with a left-invariant
bracket-generating distribution, a left-invariant metric, and a left Haar
measure $\mu$. Choose an adapted basis $X_1,\ldots,X_n$ of $\g$, with
$X_1,\ldots,X_k$ orthonormal in the horizontal layer, and use the same
symbols for the corresponding left-invariant vector fields. Set
\begin{equation}\label{eq:Lie-exp-chart}
  E_x(z):=x\exp_G\!\left(\sum_{a=1}^n z_aX_a\right).
\end{equation}
Normalize $\mu$ so that its density in exponential coordinates equals one
at the origin. If $\widehat B$ is the unit tangent ball, set
\begin{equation}\label{eq:Lie-Mij}
  \widehat V=|\widehat B|,
  \qquad
  M_{ij}=\int_{\widehat B}\xi_i\xi_j\,\dd\xi.
\end{equation}
\begin{theorem}[Left-invariant Lie-group case]
\label{thm:Lie-formula}
For every $f\in C_c^\infty(G)$, locally uniformly in $x$,
\begin{equation}\label{eq:Lie-limit-formula}
  A_hf(x)\longrightarrow\mathcal A_Gf(x),
\end{equation}
where
\begin{equation}\label{eq:Lie-limit-operator}
  \mathcal A_Gf(x)
  =\frac1{\widehat V}\left(
    \frac12\sum_{i,j=1}^kM_{ij}X_iX_jf(x)
    -\frac12\sum_{i,j=1}^k
      M_{ij}\Tr(\ad_{X_j})X_if(x)
  \right).
\end{equation}
In particular, the drift vanishes when $G$ is unimodular.
\end{theorem}
\begin{proof}
The left-invariant fields associated with the adapted basis form a global
adapted bracket frame, and \eqref{eq:Lie-exp-chart} is precisely its
first-kind chart. Left invariance gives
\[
  B(x,h)=xB(e,h),\qquad |B(x,h)|=|B(e,h)|.
\]
Let $Z_z=\sum_a z_aX_a$ and set
\begin{equation}\label{eq:Lie-rescaled-ball}
  B_h:=\delta_{1/h}\bigl(E_e^{-1}(B(e,h))\bigr).
\end{equation}
The rescaled ball in the chart $E_x$ is the same set $B_h$ for every $x$.
After shrinking the exponential neighborhood, if necessary, to make it
invariant under $z\mapsto-z$, distance symmetry and left invariance give
\[
  d(e,\exp_G(-Z_z))
  =d(\exp_G(Z_z),e)
  =d(e,\exp_G(Z_z)).
\]
Thus $B_h=-B_h$ for all sufficiently small $h$, and
\begin{equation}\label{eq:Lie-centered-balls}
  \int_{B_h}\xi_a\,\dd\xi=0,
  \qquad a=1,\ldots,n.
\end{equation}
In exponential coordinates, left Haar measure has density
\[
  J(z)=\left|\det\left(
    \frac{1-e^{-\ad_{Z_z}}}{\ad_{Z_z}}
  \right)\right|.
\]
The determinant is positive near the origin, and
\[
  \frac{I-e^{-A}}A=I-\frac12A+O(A^2).
\]
It follows that
\[
  \beta_j=(\partial_{z_j}\log J)(0)
  =-\frac12\Tr(\ad_{X_j}).
\]
On a fixed compact set containing all $B_h$,
\[
  J(\delta_h\xi)
  =
  1-\frac h2\sum_{j=1}^k
    \Tr(\ad_{X_j})\xi_j+O(h^2).
\]
Together with \eqref{eq:Lie-centered-balls} and tangent-ball moment
convergence, this gives
\[
\begin{split}
  \frac1h\int_{B_h}\xi_iJ(\delta_h\xi)\,\dd\xi
  &=-\frac12\sum_{j=1}^k\Tr(\ad_{X_j})
    \int_{B_h}\xi_i\xi_j\,\dd\xi+O(h)\\
  &\longrightarrow
  -\frac12\sum_{j=1}^kM_{ij}\Tr(\ad_{X_j}).
\end{split}
\]
Theorem~\ref{thm:general-first-moment} now gives
\eqref{eq:Lie-limit-operator}. The expansion above, as well as the Taylor
expansion in the proof of that theorem, is uniform for $x$ in compact sets;
hence the convergence is locally uniform. The sign of the drift reflects
the use of left-invariant fields and left Haar measure.
\end{proof}
The corresponding formula for a smoothly weighted Haar measure follows
immediately.
\begin{corollary}[Smoothly weighted Haar measure]
\label{cor:Lie-weighted-formula}
Let $W\in C^\infty(G)$ and $\mu_W=e^W\mu$, where $\mu$ is left Haar
measure, and write $A_h^{\mu_W}$ for \eqref{eq:AMV-def} formed with
$\mu_W$. Then, for every $f\in C_c^\infty(G)$, locally uniformly in $x$,
\begin{equation}\label{eq:Lie-weighted-limit}
\begin{split}
  A_h^{\mu_W}f(x)\longrightarrow\frac1{\widehat V}\Bigg[&
    \frac12\sum_{i,j=1}^kM_{ij}X_iX_jf(x)\\
  &+\sum_{i,j=1}^kM_{ij}
    \left(X_jW(x)-\frac12\Tr(\ad_{X_j})\right)X_if(x)
  \Bigg].
\end{split}
\end{equation}
\end{corollary}
\begin{proof}
Apply Corollary~\ref{cor:smooth-measure-change} to
Theorem~\ref{thm:Lie-formula}. Here
$C_{ij}=M_{ij}/\widehat V$ is constant in the left-invariant frame, so
\eqref{eq:smooth-measure-change} gives exactly
\eqref{eq:Lie-weighted-limit}. Local uniformity follows from the uniform
density and Taylor expansions on compact subsets of $G$.
\end{proof}

\begin{remark}[Local ball reversal]
The preceding argument is not specific to Lie groups. Its essential ingredient is the existence, at each point \(x\), of a first-kind privileged chart \(E_x\) satisfying the local ball-reversal property
\[
d\bigl(x,E_x(-z)\bigr)=d\bigl(x,E_x(z)\bigr)
\]
for \(z\) sufficiently small. The coordinate images of small balls are then centrally symmetric, so their unweighted first moments vanish identically. Theorem~\ref{thm:general-first-moment} therefore applies, and the dependence on a smooth positive measure is again determined by its coordinate density.
The same mechanism applies to locally symmetric sub-Riemannian structures and, more generally, to equiregular sub-Riemannian homogeneous spaces. Indeed, if \(M=G/H\), \(o=eH\), and \(G\) acts by sub-Riemannian isometries, then suitable exponential coordinates satisfy
\[
d\bigl(o,\exp_G(-Z)\cdot o\bigr)
=d\bigl(\exp_G(Z)\cdot o,o\bigr)
=d\bigl(o,\exp_G(Z)\cdot o\bigr).
\]
Translation by the \(G\)-action gives the corresponding reversal at every point. Thus the pointwise convergence results above extend to this setting, with the limiting operator obtained from the same moment and density expansions. We do not pursue the resulting formulas here.
\end{remark}

\subsection{Distributional convergence}
\label{subsec:detailed-balance}
On an adapted-frame patch set
\[
  M^\sharp(df):=\sum_{i,j=1}^kM_{ij}(X_jf)X_i.
\]
By Lemma~\ref{lem:intrinsic-moments}, these local expressions agree on
overlaps and, for each smooth $f$, define a global continuous horizontal
vector field.
We identify locally integrable functions with distributions using $\mu$.
For a locally integrable vector field $Y$, our convention is
\[
  \langle\operatorname{div}_\mu Y,\phi\rangle
  :=-\int_M Y\phi\,\dd\mu,
  \qquad \phi\in C_c^\infty(M).
\]
Multiplication by $D_h$ removes the center-dependent volume normalization
in $A_h$ and reveals an exact symmetric form.
\begin{proposition}[Distributional convergence]
\label{prop:detailed-balance-convergence}
For $f,g\in C_c^\infty(M)$, the identity
\begin{equation}\label{eq:finite-detailed-balance}
  \int_M fA_hg\,D_h\,\dd\mu
  =-\frac1{2h^{Q+2}}
  \iint_{\{d(x,q)<h\}}
  (f(q)-f(x))(g(q)-g(x))
  \,\dd\mu(q)\dd\mu(x)
\end{equation}
holds for all  $h>0$ small enough\footnote{Note that our general assumptions do not a priori imply $\mu(B(x,h))<+\infty$ for every $h>0$.}. Moreover, for every
$f\in C_c^\infty(M)$,
\begin{equation}\label{eq:detailed-balance-convergence}
  D_hA_hf
  \longrightarrow
  \frac12\operatorname{div}_{\mu}\bigl(M^\sharp(df)\bigr)
\end{equation}
in distributions.
\end{proposition}
\begin{proof}
Fix $f,g\in C_c^\infty(M)$. By \eqref{eq:AMV-def} and
\eqref{eq:Dh-definition-general},
\[
  \int_M fA_hg\,D_h\,\dd\mu
  =\frac1{h^{Q+2}}\iint_{\{d(x,q)<h\}}f(x)\bigl(g(q)-g(x)\bigr)
  \,\dd\mu(q)\dd\mu(x).
\]
Exchanging $x$ and $q$ and averaging gives
\eqref{eq:finite-detailed-balance}. We then prove the distributional convergence. Choose a relatively compact open set
$U\supset\operatorname{supp}f\cup\operatorname{supp}g$ and $h_0>0$ such
that
$d(\operatorname{supp}f\cup\operatorname{supp}g,M\setminus U)>h_0$.
Put $L:=\overline U$. Cover $L$ by finitely many open sets
$O_1,\ldots,O_m$ carrying adapted bracket frames as in
Subsection~\ref{subsec:general-definitions}, and choose
$\chi_l\in C_c^\infty(O_l)$ with $\sum_l\chi_l=1$ on a neighborhood of
$L$. After decreasing $h_0$, the change of variables
$q=\Phi_x(\delta_h\xi)$ is valid for
$x\in\operatorname{supp}\chi_l$ and $0<h<h_0$. Write
$\Delta_hf(x,\xi):=f(\Phi_x(\delta_h\xi))-f(x)$ and
$\Delta_hg$ analogously. The contribution of the $l$th patch to
\eqref{eq:finite-detailed-balance} is
\[
  -\frac12\int_M\chi_l(x)
  \left[\int_{B_{x,h}}
  \frac{\Delta_hf(x,\xi)}h\,\frac{\Delta_hg(x,\xi)}h
  \,J_x(\delta_h\xi)\,\dd\xi\right]\dd\mu(x),
\]
where $\dd\mu(q)=h^QJ_x(\delta_h\xi)\dd\xi$. By
\eqref{eq:Taylor-general},
\[
  \frac{\Delta_hf(x,\xi)}h\frac{\Delta_hg(x,\xi)}h
  =\sum_{i,j=1}^k\xi_i\xi_jX_if(x)X_jg(x)+O(h),
\]
uniformly for $x\in\operatorname{supp}\chi_l$ and $\xi$ in the common
compact box $\mathcal C_l$ furnished by
Proposition~\ref{prop:tangent-ball-convergence-general}. Uniform compact
containment and boundedness of $J_x(\delta_h\xi)$ make the integrated
remainder $O(h)$. The same proposition, applied to the displayed
quadratic term with $\theta=1$, therefore gives, uniformly for
$x\in\operatorname{supp}\chi_l$,
\[
  \int_{B_{x,h}}
  \frac{\Delta_hf}{h}\frac{\Delta_hg}{h}
  J_x(\delta_h\xi)\,\dd\xi
  \longrightarrow
  \sum_{i,j=1}^kM_{ij}(x)X_if(x)X_jg(x).
\]
Integrating against $\chi_l$ and summing over $l$ yields
\[
   \int_M fA_hg\,D_h\,\dd\mu
  =\int_M gA_hf\,D_h\,\dd\mu
  \longrightarrow \left\langle\frac12\operatorname{div}_\mu
    \bigl(M^\sharp(df)\bigr),g\right\rangle.
\]
\end{proof}
We call a family $(T_h)$ \emph{locally bounded on test functions} if, for
every $f\in C_c^\infty(M)$ and $K\Subset M$, there are $h_{K,f}>0$ and
$C_{K,f}<\infty$ such that
\[
  \sup_{x\in K,\ 0<h<h_{K,f}}|T_hf(x)|\le C_{K,f}.
\]
\begin{corollary}[Identity for a pointwise limit]
\label{cor:pointwise-detailed-balance}
Suppose that $(A_h)$ is locally bounded on test functions and that
$A_hf\to\mathcal A_\mu f$ pointwise for every
$f\in C_c^\infty(M)$. Then
\begin{equation}\label{eq:pointwise-detailed-balance}
  \widehat V\,\mathcal A_\mu f
  =\frac12\operatorname{div}_{\mu}\bigl(M^\sharp(df)\bigr)
\end{equation}
in distributions.
\end{corollary}
\begin{proof}
Let $\phi\in C_c^\infty(M)$. The pointwise limit
$\mathcal A_\mu f$ inherits the local bound on $A_hf$. Since
$D_h\to\widehat V$ locally uniformly, dominated convergence gives
\[
  \int_M\phi D_hA_hf\,\dd\mu
  \longrightarrow
  \int_M\phi\widehat V\mathcal A_\mu f\,\dd\mu.
\]
The conclusion follows by comparing this limit with
\eqref{eq:detailed-balance-convergence}.
\end{proof}
\section{The case of step at most two}
\label{sec:step-two}
Throughout this section, $(M,\D,g)$ is smooth and equiregular of step
$s\le2$, and $\mu$ is a smooth positive measure; the Riemannian case
$\D=TM$ is included.  The step assumption
enters through the fact that every nilpotentization is a Carnot group of
step at most two.  Consequently, the tangent frames are two-generating,
uniformly on compact parameter sets, and the centered localization and
stability estimates of Lemmas~\ref{lem:uniform-confinement} and
\ref{lem:app-distance-stability} apply.
On a fixed first-kind privileged-coordinate patch, set
\[
  m_i(x,h):=\int_{B_{x,h}}\xi_i\,\dd\xi,
  \qquad 1\le i\le k.
\]
\begin{theorem}[Pointwise AMV limits in step at most two]\label{thm:main}
Let $(M,\D,g)$ be a smooth equiregular sub-Riemannian manifold of step
$s\le2$, endowed with a smooth positive measure $\mu$.  For every $x\in M$,
every system of first-kind privileged coordinates associated with a local
adapted bracket frame, and every $1\le i\le k$, the finite limit
\begin{equation}\label{eq:Lambda-def}
  \Lambda_i(x):=\lim_{h\downarrow0}\frac{m_i(x,h)}h
\end{equation}
exists.  Consequently, for every $f\in C_c^\infty(M)$,
\begin{equation}\label{eq:main-limit}
  A_hf(x)\longrightarrow\mathcal A_\mu f(x),
\end{equation}
where $\mathcal A_\mu$ is given by
\eqref{eq:general-limit-operator} and
\begin{equation}\label{eq:Gamma-formula}
  \Gamma_i(x)=J_x(0)\Lambda_i(x)
  +\sum_{j=1}^kM_{ij}(x)\beta_j(x).
\end{equation}
\end{theorem}
\begin{remark}
The assertion about the first moments is pointwise.  The proof does not
establish continuity of $\Lambda_i$ or $\Gamma_i$, or locally uniform
convergence of $m_i(x,h)/h$.  Nevertheless, the uniform centering estimate
of Proposition~\ref{prop:boundedness} yields local boundedness and local
$L^p$ convergence; see Corollary~\ref{cor:step-two-consequences}.
\end{remark}
The argument reduces the convergence problem to one-dimensional tame
geometry.  Proposition~\ref{prop:boundedness} gives
$m_i(x,h)=O(h)$.  Lemma~\ref{lem:finite-jet-ball-comparison} then replaces
$B_{x,h}$, with an $o(h)$ error in measure, by the model family whose global
subanalyticity is given by Lemma~\ref{lem:model-family-subanalytic}.  Tame
integration and o-minimal monotonicity exclude oscillation of the normalized
moments.
\subsection{Tame geometry}
\label{subsec:tame-geometry}
We begin with the three facts from tame geometry used below.  For
background, see \cite{BierstoneMilman1988,vdD1998,CluckersMiller2011} and Section 4 in \cite{LeDonneLerarioNalonPaddeuRizzi2026}.
A set or function is \emph{globally subanalytic} if it is definable in
$\R_{\mathrm{an}}$, the real field expanded by restricted analytic
functions.  On a globally subanalytic set, a \emph{constructible function}
is a finite sum of terms
\[
  f\prod_{\nu=1}^m\log g_\nu,
\]
where $f$ and the positive functions $g_\nu$ are globally subanalytic.
Constructible functions form an algebra and are definable in
$\R_{\mathrm{an},\exp}$, which is the smallest o-minimal structure obtained by adjoining the exponential function to $\R_{\mathrm{an}}$. 
\begin{theorem}[Tame integration]\label{thm:tame-integration}
Let $X\subset\R^\ell$ be globally subanalytic and let $F$ be a
constructible function on $X\times\R^n$.  If
$\xi\mapsto F(t,\xi)$ is Lebesgue integrable for every $t\in X$, then
\[
  t\longmapsto\int_{\R^n}F(t,\xi)\,\dd\xi
\]
is constructible on $X$.
\end{theorem}
\begin{proof}
This is \cite[Theorem~1.3]{CluckersMiller2011}.
\end{proof}
\begin{lemma}[Bounded constructible limits]
\label{lem:definable-calculus}
Let $\psi:(0,\varepsilon)\to\R$ be bounded and constructible.  Then
$\lim_{t\downarrow0}\psi(t)$ exists and is finite.
\end{lemma}
\begin{proof}
The function is definable in the o-minimal structure
$\R_{\mathrm{an},\exp}$ \cite[Corollary~5.13]{vdDMM1994}.  The Monotonicity
Theorem \cite[Chapter~3]{vdD1998} makes it eventually monotone, and
boundedness gives a finite limit.
\end{proof}
\begin{lemma}[Tame Lipschitz bound]
\label{lem:tame-lipschitz}
Let $\psi:(0,\varepsilon)\to\R$ be constructible.  If
$\psi(t)=a+O(t)$ as $t\downarrow0$ for some $a\in\R$, then, after
decreasing $\varepsilon$, the function $\psi$ is Lipschitz.
\end{lemma}
\begin{proof}
By o-minimal monotonicity in its $C^1$ form
\cite[Chapter~3]{vdD1998}, after decreasing $\varepsilon$ the function
$\psi$ is $C^1$ and $\psi'$ is monotone.  Thus
$\lambda:=\lim_{t\downarrow0}\psi'(t)$ exists in the extended real line.
If $|\lambda|=\infty$, the mean value theorem on $[t/2,t]$ gives some
$\tau_t\in(t/2,t)$ for which
\[
  \frac{|\psi(t)-\psi(t/2)|}{t}=\frac{|\psi'(\tau_t)|}2
  \longrightarrow\infty,
\]
contradicting $\psi(t)=a+O(t)$.  Hence $\lambda$ is finite, so $\psi'$ is
bounded near the origin and $\psi$ is Lipschitz there.
\end{proof}
\subsection{Linear trapping estimate}
\label{subsec:boundedness}
Fix an adapted bracket frame on $O$ as in
Subsection~\ref{subsec:general-definitions}, and let $K\Subset O$.  For
$h>0$, write
\[
  X^{x,h}:=(X_1^{x,h},\ldots,X_k^{x,h}),
\]
where the rescaled fields are defined in
\eqref{eq:general-rescaled-fields}.  We write
$\widehat X_{i,x}:= X_i^{x,0} $.
\begin{proposition}[Expansion of the rescaled frame]
\label{prop:frame-expansion}
Let $\Omega\subset\R^n$ be a compact coordinate box.  For every
$1\le i\le k$, the coefficients of $X_i^{x,h}$ extend smoothly in
$(x,h,\xi)$ to a neighborhood of
$K\times\{0\}\times\Omega$.  Moreover, for every integer $m\ge0$, there
exist $C_{K,\Omega,m}>0$ and $h_{K,\Omega,m}>0$ such that
\begin{equation}\label{eq:frame-rate}
  \sup_{x\in K}
  \bigl\|X_i^{x,h}-\widehat X_{i,x}\bigr\|_{C^m(\Omega)}
  \le C_{K,\Omega,m}h
\end{equation}
for every $h \in (0,h_{K,\Omega,m})$.
\end{proposition}
\begin{proof}
In first-kind coordinates, write
\[
  (\Phi_x^{-1})_*X_i
  =\sum_{a=1}^nf_a^i(x,z)\partial_{z_a}.
\]
For a multi-index $\alpha=(\alpha_1,\ldots,\alpha_n)$, its weighted degree is $\sum_{b=1}^n w_b\alpha_b$.
Privilegedness implies that the Taylor coefficients of
$f_a^i(x,\cdot)$ of weighted degree less than $w_a-1$ vanish.  Hence the $\partial_{\xi_a}$-coefficient
\begin{equation}\label{eq:rescaled-coefficient}
  F_a^i(x,h,\xi)
  :=h^{1-w_a}f_a^i(x,\delta_h\xi)
\end{equation}
extends smoothly through $h=0$.  Its value at
$h=0$ is the corresponding coefficient of $\widehat X_{i,x}$.  Taylor's
formula in $h$ then
gives
\[
  F_a^i(x,h,\xi)=F_a^i(x,0,\xi)+hR_a^i(x,h,\xi),
\]
where $R_a^i$ is uniformly bounded in $C_\xi^m$ on the prescribed compact
set.  This proves
\eqref{eq:frame-rate}.
\end{proof}
We now improve the tangent-ball trapping of
Proposition~\ref{prop:uniform-ball-control} to a linear estimate.
\begin{proposition}[Linear trapping and first-moment bound]
\label{prop:boundedness}
For every $K\Subset O$, there are constants $C,h_K>0$ such that
\begin{equation}\label{eq:linear-trapping}
  \widehat B_x(1-Ch)\subset B_{x,h}
  \subset\widehat B_x(1+Ch),
  \qquad x\in K,\quad 0<h<h_K,
\end{equation}
and
\begin{equation}\label{eq:moment-bound}
  \left|\int_{B_{x,h}}\xi_i\,\dd\xi\right|\le Ch,
  \qquad 1\le i\le k.
\end{equation}
\end{proposition}
\begin{proof}
By Proposition~\ref{prop:uniform-ball-control}, there is a compact
coordinate box $\mathcal C$ containing $B_{x,h}$ and $\widehat B_x(2)$ for
every $x\in K$ and all sufficiently small $h$.  Choose a compact coordinate
box $\mathcal C^+$ such that
$\mathcal C\cup\{0\}\subset\operatorname{int}\mathcal C^+$, and apply
Lemma~\ref{lem:uniform-confinement} with $ C=\mathcal C^+$ to the
family $X^{x,h}$.  The lemma supplies common boxes and the same confinement
bound for $X^{x,h}$ and $\widehat X_x$.  Applying
Lemma~\ref{lem:app-distance-stability} and
Proposition~\ref{prop:frame-expansion}, and decreasing $h_K$ if necessary,
yields
\begin{equation}\label{eq:squared-distance-Ch}
  \sup_{\substack{x\in K\\\xi\in\mathcal C}}
  \left|d_{x,h}(0,\xi)^2-\widehat d_x(0,\xi)^2\right|
  \le C_0h.
\end{equation}
Choose $C>C_0/2$ and decrease $h_K$ so that $1-Ch>0$ and
\[
  1+C_0h<(1+Ch)^2,
  \qquad
  (1-Ch)^2+C_0h<1.
\]
If $\xi\in B_{x,h}$, then \eqref{eq:squared-distance-Ch} gives
$\widehat d_x(0,\xi)<1+Ch$.  If
$\widehat d_x(0,\xi)<1-Ch$, the same estimate gives
$d_{x,h}(0,\xi)<1$.  This proves \eqref{eq:linear-trapping}. As already pointed out, in first-kind coordinates, inversion on the tangent group is
$\xi\mapsto-\xi$.  It preserves the distance from the identity and Lebesgue
measure.  Thus $\widehat B_x$ is centrally symmetric and
$\int_{\widehat B_x}\xi_i\,\dd\xi=0$.  By
\eqref{eq:linear-trapping}, homogeneity, and the uniform bounds for
$|\widehat B_x|$ and $\sup_{\mathcal C}|\xi_i|$, we obtain
\begin{equation}\label{eq:ball-difference-linear}
  \left|\int_{B_{x,h}}\xi_i\,\dd\xi\right|
  \le \sup_{\mathcal C}|\xi_i|\,
      |B_{x,h}\mathbin\triangle\widehat B_x|
  \le C\bigl((1+Ch)^Q-(1-Ch)^Q\bigr)\le C'h.
\end{equation}
This is \eqref{eq:moment-bound}.
\end{proof}
\subsection{Analytic finite-jet reduction}
\label{subsec:analytic-comparison}
Fix $x\in O$.  All constants in this subsection may depend on $x$.  We
approximate the smooth  frame by a finite weighted
Taylor polynomial.  The resulting approximating sub-Riemannian structure is real analytic and
two-generating, so its distance is subanalytic near the diagonal by
\cite[Theorem~4.2]{AgrachevSarychev1999}; see also \cite{Jacquet1999}.
Choose compact boxes $C,C^+\subset\R^n$ such that
\[
  \overline{\widehat B_x(3)}\subset\operatorname{int}C,
  \qquad C\subset\operatorname{int}C^+.
\]
Apply Lemma~\ref{lem:uniform-confinement} with $K=\{x\}$ and $C=C^+$.
Let $U_0\Subset U_1$, $R$, and $h_*>0$ be the resulting
localization data for $X^{x,h}$.  After decreasing $h_*$, the frames
$X^{x,h}$ are pointwise linearly independent and two-generating on
$\overline U_1$ for $0\le h\le h_*$.  Proposition~\ref{prop:boundedness}
and the identity
\[
  \{\xi:d_{x,h}(0,\xi)<2\}=\delta_2B_{x,2h}
\]
show, after a further decrease of $h_*$, that the centered radius-two
$X^{x,h}$-balls are contained in $\widehat B_x(3)$ and hence in
$\operatorname{int}C$.
In first-kind coordinates, write
\[
  (\Phi_x^{-1})_*X_i \, (z)
  =\sum_{a=1}^n f_a^i(z)\partial_{z_a}.
\]
where the $f_a^i$ are smooth functions. For a multi-index $\alpha=(\alpha_1,\ldots,\alpha_n)$, write
$|\alpha|_w:=\sum_{b=1}^n w_b\alpha_b$ and $\alpha ! = \alpha_1!\ldots \alpha_n!$.
Fix an integer $N\ge3$.  For every $i$ and $a$, let $f_a^{i,[N]}$ be the
weighted Taylor polynomial
\[
  f_a^{i,[N]}(z)
  :=\sum_{|\alpha|_w\le w_a+N-2}
    \frac{\partial^\alpha f_a^i(0)}{\alpha!}z^\alpha.
\]
Define the polynomial vector fields
\[
  Y_i:=\sum_{a=1}^n f_a^{i,[N]}\partial_{z_a},
  \qquad 1\le i\le k,
\]
and write $Y:=(Y_1,\ldots,Y_k)$.  Its rescaled frame
$Y^h:=(Y_1^h,\ldots,Y_k^h)$ is defined componentwise by
\[
  Y_i^h:=h(\delta_{1/h})_*Y_i.
\]
Fix an arbitrary homogeneous quasi-norm $\|\cdot\|_w$ satisfying
$\|\delta_hz\|_w=h\|z\|_w$.
\begin{lemma}[Weighted Taylor estimate]
\label{lem:weighted-taylor-estimate}
 There are a neighborhood $W$ of the origin
and a constant $C_N$ such that, for $z\in W$, $1\le i\le k$,
$1\le a\le n$, and every multi-index $\gamma$ with $|\gamma|_w\le2$,
\[
  \bigl|\partial_z^\gamma
    (f_a^i-f_a^{i,[N]})(z)\bigr|
  \le C_N\|z\|_w^{w_a+N-1-|\gamma|_w}.
\]
Consequently, after decreasing $h_*$ so that
$\delta_h\overline U_1\subset W$ for $0<h<h_*$,
\begin{equation}\label{eq:finite-jet-frame-error}
  \max_{1\le i\le k}
  \|X_i^{x,h}-Y_i^h\|_{C^2(\overline U_1)}
  \le C_Nh^N,
  \qquad 0<h<h_*.
\end{equation}
\end{lemma}
\begin{proof}
We first record the scalar weighted Taylor estimate.  Let $q\ge2$, let
$\varphi$ be of class $C^{q+1}$ near the origin, and set
\[
  P_q\varphi(z):=\sum_{|\alpha|_w\le q}
    \frac{\partial^\alpha\varphi(0)}{\alpha!}z^\alpha.
\]
Since $q-|\gamma|_w\ge0$, we have
$\partial^\gamma P_q\varphi=P_{q-|\gamma|_w}(\partial^\gamma\varphi)$.
Write $z=\delta_t\theta$, where $t=\|z\|_w$ and
$\|\theta\|_w=1$.  For $0\le j\le q-|\gamma|_w$,
\[
  \frac1{j!}\left.\frac{d^j}{d s^j}\right|_{s=0}
  \partial^\gamma\varphi(\delta_s\theta)
  =\sum_{|\beta|_w=j}
    \frac{\partial^{\beta+\gamma}\varphi(0)}{\beta!}\theta^\beta.
\]
Consequently, Taylor's formula in the single variable $s$, applied to
$s\mapsto\partial^\gamma\varphi(\delta_s\theta)$, gives, uniformly for
$\theta$ on the compact weighted unit sphere,
\[
  \left|\partial^\gamma(\varphi-P_q\varphi)(z)\right|
  \le Ct^{q+1-|\gamma|_w}.
\]
Taking
$q=w_a+N-2\ge2$ and $\varphi=f_a^i$ proves the first assertion.
The $\partial_{\xi_a}$-coefficient of
$X_i^{x,h}-Y_i^h$ is
\[
  h^{1-w_a}(f_a^i-f_a^{i,[N]})(\delta_h\xi).
\]
For each such $\gamma$, the chain rule gives
\[
  \partial_\xi^\gamma
  \left[h^{1-w_a}
    (f_a^i-f_a^{i,[N]})(\delta_h\xi)\right]
  =h^{1-w_a+|\gamma|_w}
    \partial_z^\gamma(f_a^i-f_a^{i,[N]})(\delta_h\xi).
\]
Since $|\gamma|_w\le w_a+N-1$, the first assertion and
$\|\delta_h\xi\|_w=h\|\xi\|_w$ give the factor
\[
  h^{1-w_a+|\gamma|_w}
  h^{w_a+N-1-|\gamma|_w}=h^N.
\]
Since $\overline U_1$ is
compact, the stated $C^2$ estimate follows after decreasing $h_*$ so that
$\delta_h\overline U_1\subset W$.
\end{proof}
Since $w_a\le2$ and $N\ge3$, the truncation retains all constant and linear
monomials.  Thus the values and first derivatives of $Y_i$ at the origin
agree with those of $(\Phi_x^{-1})_*X_i$.  In particular, the values of the
$Y_i$ and their first brackets span at the origin, and the same rank
conditions hold near the origin.  The terms of weighted degree $w_a-1$ are
also retained, so the nilpotent approximation of $Y$ is $\widehat X_x$.
Choose a connected open neighborhood \(V\subset\mathbb R^n\) of the
origin on which \(Y_1,\ldots,Y_k\) are pointwise linearly independent
and two-generating. Since the \(Y_i\) are polynomial, they define a
real-analytic sub-Riemannian structure on \(V\). Let
$d_Y$ denote the associated control distance with trajectories
restricted to $V$.  After decreasing $h_*$, assume that
$\delta_h\overline{U_1}\subset V$ for $0<h<h_*$.  Define
\[
  d_{Y,h}(\eta,\zeta)
  :=h^{-1}d_Y(\delta_h\eta,\delta_h\zeta).
\]
The retained terms of weighted degree $w_a-1$ form the nilpotent
approximation, whereas every other retained term has degree $d\ge w_a$ and
rescales by $h^{1+d-w_a}=O(h)$.  Hence
$\|Y_i^h-\widehat X_{i,x}\|_{C^1(\overline U_1)}=O(h)$.
Extend the family to $h=0$ by setting
$Y^0:=\widehat X_x$.  Since the boxes $U_0\Subset U_1$ and the bound
$R$ in Lemma~\ref{lem:uniform-confinement} depend only on $C^+$ and the
tangent frame, the lemma applies to $Y^h$ with the same data after a
further decrease of $h_*$.
\begin{lemma}[Global subanalyticity of $T_*$]
\label{lem:model-family-subanalytic}
Set
\[
  T_*:=\{(h,\xi)\in(0,h_*)\times C:
       d_Y(0,\delta_h\xi)<h\}.
\]
After decreasing $h_*$, the set $T_*$ is globally subanalytic.
\end{lemma}
\begin{proof}
The polynomial frame $Y$ defines a real-analytic, two-generating
sub-Riemannian structure on $V$.  By standard arguments, the germ of $d_Y$ at the origin is the germ of the distance of a real-analytic sub-Riemannian manifold, so that there is a
neighborhood $W_0\times W_0 \subset V\times V$ of $(0,0)$ on which
$d_{Y}^2$ is subanalytic, see \cite[Theorem~4.2]{AgrachevSarychev1999}.  Pullback by the analytic map $z\mapsto(0,z)$ shows in
particular that $z\mapsto d_Y(0,z)^2$ is subanalytic on $W_0$. Choose a compact box \(K_0\subset W_0\) containing the origin in its
interior.  Since \(d_Y(0,\cdot)^2\) is continuous, its graph
\[
  \mathcal G_0:=
  \bigl\{(z,s)\in K_0\times\mathbb R:
          s=d_Y(0,z)^2\bigr\}
\]
is compact and subanalytic.  It is therefore globally subanalytic by
\cite[\S2.5(4)]{vdDMiller1996}.  The corresponding strict epigraph is
\[
  \mathcal E_0
  =\bigl\{(z,t)\in K_0\times\mathbb R:
       \text{there exists }s<t\text{ with }(z,s)\in\mathcal G_0\bigr\}.
\]
It is globally subanalytic by closure under finite intersections and
coordinate projections.  Decrease $h_*$ so that
\[
  \delta_hC\subset K_0,
  \qquad 0<h<h_*.
\]
Since the map
\[
  (h,\xi)\longmapsto(\delta_h\xi,h^2)
\]
is polynomial, we have
\[
  T_*=
  \bigl\{(h,\xi)\in(0,h_*)\times C:
  (\delta_h\xi,h^2)\in\mathcal E_0\bigr\}.
\]
Closure of globally subanalytic sets under definable inverse images and
finite intersections proves the claim.
\end{proof}
\begin{lemma}[Finite-jet ball comparison]
\label{lem:finite-jet-ball-comparison}
For $0<h<h_*$, set
\[
  T_h:=\{\xi\in C:(h,\xi)\in T_*\}.
\]
There exists $h_0\in(0,h_*)$ such that, for $0<h<h_0$,
$T_h=\{\xi\in\R^n:d_{Y,h}(0,\xi)<1\}$ and
\begin{equation}\label{eq:finite-jet-ball-comparison}
  |B_{x,h}\mathbin\triangle T_h|=o(h)
  \qquad\text{as }h\downarrow0.
\end{equation}
\end{lemma}
\begin{proof}
Repeating the argument for the proof of Lemma \ref{lem:uniform-confinement}, point 4,  we see that for every
$\xi\in C^+$,
\begin{equation}\label{eq:model-confined-distance}
  d_{Y,h}(0,\xi)=d_{Y^h,U_1}(0,\xi).
\end{equation}
The same first-exit argument applies to every control of norm less than two.
Comparing the corresponding $Y^h$- and $\widehat X_x$-trajectories by
Gronwall's inequality shows, after a further decrease of $h_*$, that for
every $0<h<h_*$ the centered radius-two model ball is contained in
$\widehat B_x(3)$ and hence in
$\operatorname{int}C$.  It follows from the definition of $d_{Y,h}$ that
\[
  \{\xi\in\R^n:d_{Y,h}(0,\xi)<1\}=\{\xi\in C:d_{Y,h}(0,\xi)<1\}.
\]
Theorem~\ref{thm:tame-integration}, applied to the indicator of $T_*$
extended by zero outside $C$, therefore shows that
$h\mapsto|T_h|$ is constructible on $(0,h_*)$.
Since $Y^h=\widehat X_x+O(h)$ in $C^1(\overline U_1)$,
Lemma~\ref{lem:app-distance-stability}, applied with
$(C,C')=(C,C^+)$ gives
\[
  \sup_{\xi\in C}
  |d_{Y,h}(0,\xi)^2-\widehat d_x(0,\xi)^2|\le Ch.
\]
As in Proposition~\ref{prop:boundedness}, this implies
\[
  \widehat B_x(1-Ch)
  \subset\{\xi:d_{Y,h}(0,\xi)<1\}
  \subset\widehat B_x(1+Ch),
\]
and hence
\[
  |T_h|=|\widehat B_x|+O(h).
\]
Lemma~\ref{lem:tame-lipschitz} therefore makes $h\mapsto|T_h|$ Lipschitz
near $0$.
Let $\varepsilon_0$ be the smallness threshold in
Lemma~\ref{lem:app-distance-stability}.  Choose $L<\infty$ and $h_0>0$ so
that $2h_0<h_*$, the function $h\mapsto|T_h|$ is $L$-Lipschitz on
$(0,2h_0)$, and
\[
  C_Nh^N\le\min\{\tfrac12,\varepsilon_0\},
  \qquad 0<h<h_0.
\]
By Lemma~\ref{lem:app-distance-stability} and
\eqref{eq:finite-jet-frame-error}, after enlarging the constant if necessary,
\[
  \sup_{\xi\in C}
  |d_{x,h}(0,\xi)^2-d_{Y,h}(0,\xi)^2|
  \le Ch^N.
\]
After decreasing $h_0$ so that $Ch^N<1$, this gives
\[
  \{\xi:d_{Y,h}(0,\xi)<\sqrt{1-Ch^N}\}
  \subset B_{x,h}
  \subset
  \{\xi:d_{Y,h}(0,\xi)<\sqrt{1+Ch^N}\}.
\]
For $r$ in a neighborhood of $1$ and $rh<h_*$, the definition of
$d_{Y,h}$ gives the exact scaling identities
\[
  d_{Y,h}(0,\delta_r\xi)=r\,d_{Y,rh}(0,\xi),
  \qquad
  |\{\xi:d_{Y,h}(0,\xi)<r\}|=r^Q|T_{rh}|.
\]
All balls occurring here are contained in $C$ by the radius-two confinement
proved above.  Since both $B_{x,h}$ and $T_h$ lie between the two balls in
the preceding inclusions, the boundedness and Lipschitz continuity of
$h\mapsto|T_h|$
give
\[
\begin{split}
  |B_{x,h}\mathbin\triangle T_h|
  &\le (1+Ch^N)^{Q/2}
      \left|T_{h\sqrt{1+Ch^N}}\right|\\
  &\quad-(1-Ch^N)^{Q/2}
      \left|T_{h\sqrt{1-Ch^N}}\right|\\
  &\le Ch^N
      +Lh\bigl(\sqrt{1+Ch^N}-\sqrt{1-Ch^N}\bigr)\\
  &=O(h^N)=o(h).
\end{split}
\]
This is \eqref{eq:finite-jet-ball-comparison}.
\end{proof}
\begin{proof}[Proof of Theorem~\ref{thm:main}]
Fix $x\in M$, choose a privileged-coordinate patch $O\ni x$, and let
$1\le i\le k$.  Apply the construction above at $x$, and let
$C,T_*,h_0$, and $T_h$ be the resulting objects, with $h_0$ and $T_h$ as
in Lemma~\ref{lem:finite-jet-ball-comparison}.  By
Lemma~\ref{lem:model-family-subanalytic}, $T_*$ is globally subanalytic.
Applying Theorem~\ref{thm:tame-integration} to
$\mathbf 1_{T_*}(h,\xi)\xi_i$, extended by zero outside $C$, shows that
\[
  \widetilde m_i(h):=\int_{T_h}\xi_i\,\dd\xi
\]
is constructible.  Since $C$ is bounded,
\eqref{eq:finite-jet-ball-comparison} gives
$m_i(x,h)-\widetilde m_i(h)=o(h)$.  Proposition~\ref{prop:boundedness}
then implies that the constructible function $\widetilde m_i(h)/h$ is
bounded near the origin.  By Lemma~\ref{lem:definable-calculus}, this
quotient has a finite limit.  The $o(h)$ comparison transfers the limit to
$m_i(x,h)/h$ and proves \eqref{eq:Lambda-def}.  Finally,
Lemma~\ref{lem:unweighted-first-moments} gives \eqref{eq:Gamma-formula}, and
Theorem~\ref{thm:general-first-moment} gives \eqref{eq:main-limit}.
\end{proof}
\subsection{Local \texorpdfstring{$L^p$}{Lp} convergence}
\begin{lemma}[Lipschitz regularity of the coefficients]
\label{lem:tangent-coefficients-Lipschitz}
On a privileged-coordinate patch of step at most two, the functions
$x\mapsto\widehat V_x$ and $x\mapsto M_{ij}(x)$ are locally Lipschitz.
\end{lemma}
\begin{proof}
Fix $K\Subset O$, where $O$ is the coordinate patch, and choose compact
boxes $\mathcal C,\mathcal C^+$ such that $\mathcal C$ contains
$\widehat B_x(2)$ for every $x\in K$ and
$\mathcal C\cup\{0\}\subset\operatorname{int}\mathcal C^+$.
Smooth dependence on $x$, uniform two-generation, and
Lemma~\ref{lem:uniform-confinement}, applied with $C=\mathcal C^+$,
provide common centered localization data.
Apply Lemma~\ref{lem:app-distance-stability} with
$(C,C')=(\mathcal C,\mathcal C^+)$ and parameter space
$\Theta=K\times K$, $Y^{(x,y)}:=\widehat X_x$ and
$Z^{(x,y)}:=\widehat X_y$, whose $C^1$-distance is $O(|x-y|)$.  Its
estimate \eqref{eq:frame-distance-stability} requires
$\max_i\|\widehat X_{i,x}-\widehat X_{i,y}\|_{C^1(\overline{U_1})}
\le\varepsilon_0$ and therefore applies once $|x-y|$ is small; it yields
\[
  \sup_{\xi\in\mathcal C}
  \left|
  \widehat d_x(0,\xi)^2-\widehat d_y(0,\xi)^2
  \right|\le C|x-y|.
\]
Put $\varepsilon=C|x-y|$.  The last estimate and homogeneity imply
\[
  \widehat B_x\subset
  \widehat B_y\bigl(\sqrt{1+\varepsilon}\bigr),
  \qquad
  \widehat B_y\subset
  \widehat B_x\bigl(\sqrt{1+\varepsilon}\bigr).
\]
The symmetric difference is therefore contained in two homogeneous annuli
of total volume at most
\[
  C\bigl((1+\varepsilon)^{Q/2}-1\bigr) \le C'|x-y|.
\]
The integrands $1$ and $\xi_i\xi_j$ are uniformly bounded on $\mathcal C$,
so their integrals over the tangent balls vary by $O(|x-y|)$.  Since
$x\mapsto J_x(0)$ is smooth, \eqref{eq:Vx} and
\eqref{eq:Mij-general} prove the claim when $|x-y|$ is small.  For pairs
with $|x-y|$ bounded below, the same estimate follows, after enlarging the
constant, from the uniform bounds in Lemma~\ref{lem:intrinsic-moments}.
\end{proof}
\begin{corollary}[Local convergence and divergence form]
\label{cor:step-two-consequences}
Under the assumptions of Theorem~\ref{thm:main}, the family $(A_h)$ is
locally bounded on test functions, and
\[
  A_hf\longrightarrow\mathcal A_\mu f
  \quad\text{in }L^p_{\mathrm{loc}}(M,\mu),
  \qquad 1\le p<\infty.
\]
Moreover,
\begin{equation}\label{eq:main-divergence}
  \widehat V\,\mathcal A_\mu f
  =\frac12\operatorname{div}_{\mu}\bigl(M^\sharp(df)\bigr)
  =\frac12\operatorname{div}_{\mu}\left(
    \sum_{i,j=1}^kM_{ij}(X_jf)X_i
  \right)
\end{equation}
in distributions and on every privileged-coordinate patch,
\begin{equation}\label{eq:main-coefficients-ae}
  \Gamma_j
  =\frac12\sum_{i=1}^k
  \bigl(X_iM_{ij}+M_{ij}\operatorname{div}_\mu X_i\bigr)
  \quad\text{almost everywhere}.
\end{equation}
\end{corollary}
\begin{proof}
The moment bound in Proposition~\ref{prop:boundedness}, the density
expansion \eqref{eq:J-expansion-general}, and Taylor's formula
\eqref{eq:Taylor-general} bound $A_hf$ uniformly on compact sets for small
$h$.  Since
\[
  \{(x,q)\in M\times M:d(x,q)<h\}
\]
is open, parameter integration shows that $x\mapsto A_hf(x)$ is Borel
measurable.  Pointwise convergence and dominated convergence give the
local $L^p$ convergence.  Corollary~\ref{cor:pointwise-detailed-balance}
gives \eqref{eq:main-divergence}.
To identify the coefficients, work on
an adapted-frame patch.  Near any point choose
$f_1,\ldots,f_k\in C_c^\infty(M)$ so that the matrix
$P=(X_jf_r)_{r,j\le k}$ is invertible.  The local boundedness just proved,
together with \eqref{eq:general-limit-operator}, shows by solving the
resulting system that $\Gamma_1,\ldots,\Gamma_k$ are locally bounded.
By Lemma~\ref{lem:tangent-coefficients-Lipschitz}, the weak Leibniz rule
applies to $M_{ij}$.  Expanding the divergence in
\eqref{eq:main-divergence} and comparing it with
\eqref{eq:general-limit-operator} gives
\[
 \sum_{j=1}^k c_jX_jf=0\quad\text{in distributions},\qquad
 c_j:=\Gamma_j-\frac12\sum_{i=1}^k
 \bigl(X_iM_{ij}+M_{ij}\operatorname{div}_\mu X_i\bigr).
\]
Apply this identity to $f_1,\ldots,f_k$.  Outside a common negligible set
it says $Pc=0$, hence $c=0$.  A countable patch cover proves
\eqref{eq:main-coefficients-ae}.
\end{proof}
\section{Further directions}
We close the paper  with two possible research directions.
\subsection{Higher-order asymptotics}
Higher-order asymptotics are interesting and should reveal curvature
invariants. The Riemannian case provides a useful insight.  Integrating the
small-sphere Pizzetti formula \cite[Theorem~2]{Willmore1980} in the radial
variable gives a fourth-order expansion for averages over geodesic balls;
see also \cite{GrayWillmore1982}.  More precisely, let $(M^n,g)$ be smooth and
$\mu=\vol_g$.  For the operator
\[
  A_hf(x):=\frac{1}{h^{2}}\fint_{B_g(x,h)}
  \bigl(f(y)-f(x)\bigr)\,\dd\vol_g(y),
\]
one has, as $h\downarrow0$,
\[
  A_hf(x)
  =\frac{1}{2(n+2)}\Delta_gf(x)
  +h^{2}\mathcal P^{\mathrm{Riem}}_{4}f(x)
  +O(h^{4}),
\]
where $\Delta_g=\operatorname{tr}_g\nabla^2$ and
\[
\begin{split}
  \mathcal P^{\mathrm{Riem}}_{4}f
  ={}&\frac{1}{8(n+2)(n+4)}\Delta_g^{2}f
  -\frac{1}{12(n+2)(n+4)}
  \bigl\langle\Ric_g,\nabla^{2}f\bigr\rangle_g\\
  &-\frac{1}{8(n+2)(n+4)}
  \bigl\langle\nabla\Scal_g,\nabla f\bigr\rangle_g
  +\frac{\Scal_g}{6(n+2)^{2}(n+4)}\Delta_gf .
\end{split}
\]
Extending those asymptotic expansions to sub-Riemannian manifolds of step two appears to be beyond the reach of the current methods. However, in dimension three, the contact normal form of
\cite{BarilariBeschastnyiLerario2020} suggests a  concrete expansion which involves  a universal weight-four operator built
from the intrinsic sub-Laplacian, the Reeb field, contact torsion and
curvature, and their horizontal derivatives; see
\cite{Agrachev1995,Agrachev1996} and \cite[Chapter~17]{ABB2020} for background on three-dimensional contact sub-Riemannian geometry. In that case, a formal argument (that we do not justify here) gives for
Popp measure,
\[
  A_h f
  = c_0\,\Delta_{\mathrm P} f + h^{2}\mathcal P_4 f + o(h^{2}),
\]
with no term of order $h$, where
\[
\mathcal P_4 f =
  c_1\,\Delta_{\mathrm P}^{2} f
  + c_2\,X_0^{2} f
  + c_3\,\operatorname{div}_{\mathrm P}\bigl(\kappa\nabla_H f\bigr)
  - c_4\,\langle\nabla_H\kappa,\nabla_H f\rangle
  + c_5\,\operatorname{div}_{\mathrm P}\bigl((J\tau)\nabla_H f\bigr).
\]
Here \(\Delta_{\mathrm P}\) is the intrinsic sub-Laplacian, $X_0$ is the
Reeb field, $\tau=T^{\nabla}(X_0,\cdot)$ is the torsion of the canonical
contact connection, $J$ is the complex structure, $\kappa$ is the Agrachev
curvature parameter,  see
\cite{Agrachev1996}, and the constants $c_i$ are universal.
We stress that at present this three-dimensional formula should be regarded as conjectural but gives a motivation to pursue further this direction.

\subsection{Beyond step two}
\label{subsec:beyond-step-two}
Two ingredients of the proof of Theorem~\ref{thm:main} used the  step two assumption.
\begin{enumerate}
\item \emph{Linear trapping.}
Proposition~\ref{prop:boundedness} uses Lipschitz stability of the squared
distance. Appendix~\ref{app:uniform-lipschitz} proves the required
compact-family estimate by combining a nonsmooth subgradient criterion with
Rifford's quantitative approximate Goh estimate.  Large proximal
subgradients give normal multipliers of index zero; approximate Goh and
two-generation then bound their terminal covectors uniformly.  In higher
step, the Goh condition no longer forces the multiplier to vanish. Already
in rank two and step at least three, regular abnormal extremals may be
locally minimizing for every choice of metric \cite{LiuSussmann1995}, and
some are rigid \cite{BryantHsu1993}.  No general Lipschitz estimate for the
squared distance is available near such curves.
\item \emph{Subanalyticity of the distance.}
Analyticity of the frame alone does not guarantee subanalytic distance
spheres in the presence of abnormal minimizers.  In the flat Martinet
model, for instance, small sub-Riemannian spheres are not subanalytic
\cite{ABCK1997}.  Without a definable
model family, the comparison moments in
Lemma~\ref{lem:finite-jet-ball-comparison} need not be constructible, so
o-minimality no longer rules out oscillation using our argument.
\end{enumerate}
Therefore our method does not readily extend to arbitrary higher steps.
For real-analytic medium-fat distributions, Agrachev and Sarychev prove
both subanalyticity of the distance off the diagonal and absence of
strictly abnormal weak minimizers \cite{AgrachevSarychev1999}.  These conclusions do not by
themselves provide the uniform subgradient estimate used below:
the Goh condition annihilates only
$\D+[\D,\D]$, which need not equal the tangent bundle for a medium-fat
structure.  Thus the present distance-stability argument does not directly
extend beyond step two.

\vspace{0.1cm}

The extension to higher-step structures remains an open question.

\appendix
\section{Uniform distance stability}
\label{app:uniform-lipschitz}
This appendix collects the compact-family estimates used in
Section~\ref{sec:step-two} and records their uniform dependence on the
frame.  The squared-distance estimate of Lemma \ref{lem:app-distance-stability} is obtained by bounding proximal
subgradients of the half squared distance with the quantitative
approximate Goh estimate of \cite[Proposition~3.7]{Rifford2023}, and then
using the nonsmooth Lipschitz criterion from
\cite[Propositions~2.2 and~2.4]{Rifford2023}.

Let $\Theta$ be a compact subset of a finite-dimensional smooth manifold. Consider bounded open boxes $U_0\Subset U_1\subset\R^n$, and a family of pointwise linearly independent $\mathbb{R}^n$-vector fields
\[
  W^\lambda=(W_1^\lambda,\ldots,W_k^\lambda)
\]
such that the mapping $(\lambda,x) \mapsto W^\lambda(x)$ is jointly smooth near
$\Theta\times\overline{U_1}$. For given $p \in U_1$, $\lambda \in \Theta$, and $u\in L^2([0,1];\R^k)$, write
$\gamma_{p,u}^\lambda$ for the solution of the Cauchy problem
\[
  \dot\gamma=\sum_{i=1}^k u_iW_i^\lambda(\gamma) \quad \text{a.e.~on $[0,1]$},
  \qquad \gamma(0)=p.
\]
The endpoint map associated with $p$ and $\lambda$ maps any control $u \in L^2([0,1];\R^k)$ to $$\operatorname{End}_p^\lambda(u):=\gamma_{p,u}^\lambda(1).$$ The control distance associated
with $W^\lambda$ and obtained by restricting horizontal curves to $U_1$ is given by
\[
  d_{W^\lambda}(p,q)^2
  :=\inf\left\{\|u\|_{L^2}^2:u\in L^2([0,1];\R^k) \text{ such that }
    \operatorname{End}_p^\lambda(u)=q \text{ and }
    \gamma_{p,u}^{\lambda}([0,1])\subset U_1\right\}
\]
for any $p,q \in U_1$.

\begin{definition}\label{def:uniform-generating}
The family $W^\lambda$ is \emph{uniformly two-generating} on a set $A$ if,
for every $(\lambda,x)\in\Theta\times A$, the linear map
\[
  \mathcal B_{\lambda,x}(a,b)
  :=\sum_i a_iW_i^\lambda(x)
  +\sum_{i<j}b_{ij}[W_i^\lambda,W_j^\lambda](x)
\]
is onto and admits a right inverse whose norm is bounded independently of
$(\lambda,x)$.
\end{definition}
\begin{lemma}[Uniform accessibility]
\label{lem:uniform-step-two-ball-box}
Let $W^\lambda$ be jointly smooth and uniformly two-generating on
$\overline{U_1}$.  If $V,V'$ are compact and
$V\Subset\operatorname{int}V'\Subset U_1$, there are $c\ge1$ and
$r_0>0$, uniform in $\lambda$, such that
whenever $x,y\in V$ and $|x-y|<r_0$, they can be joined inside $V'$ with
control norm at most $c|x-y|^{1/2}$.
\end{lemma}
\begin{proof}
For fixed $\lambda$, apply
\cite[Proposition~1.1 and Theorem~4]{NagelSteinWainger1985} to the
weighted family consisting of $W_i^\lambda$ with degree one and
$[W_i^\lambda,W_j^\lambda]$ with degree two.  Locally, the spanning
condition gives the smooth coefficient functions required there.
Applying the construction in an open set $O$ with
$V\Subset O\Subset\operatorname{int}V'$ gives the stated confinement.  The
degree-one fields and their degree-two commutators span with a uniformly
bounded right inverse, while the standard commutator paths realize a bracket
displacement of size $t^2$ at $W^\lambda$-cost $O(t)$.  It follows that
\[
  d_{W^\lambda}(x,y)\le c|x-y|^{1/2}
\]
whenever $x,y\in V$ and $|x-y|<r_0$, with the realizing trajectory
contained in $V'$.
The constants are uniform because the ball-box construction depends only on
finitely many local smooth seminorms of the fields on $V'$ and on a bound
for a right inverse of
$\mathcal B_{\lambda,x}$.  Joint smoothness, compactness of
$\Theta\times V'$, and Definition~\ref{def:uniform-generating} provide
these bounds.  A finite cover of $\Theta\times V$ then gives common
constants $c$ and $r_0$.
\end{proof}
\begin{lemma}[Uniform approximate Goh estimate]
\label{lem:uniform-approximate-Goh}
Let $X^\lambda=(X_1^\lambda,\ldots,X_k^\lambda)$ be a jointly smooth
family of vector fields on $\R^n$, all supported in one fixed compact set,
with $\lambda\in\Theta$.  Let $\mathcal A\subset\Theta\times
L^2([0,1];\R^k)$ be compact for the strong $L^2$ topology, and assume
that every $(\lambda,u)\in\mathcal A$ is energy minimizing between its
endpoints and that
\[
 \sup\{\|u\|_{L^\infty}:(\lambda,u)\in\mathcal A\}<\infty.
\]
Write $E_\lambda$ for the endpoint map from $0$, set
\[
 J(u):=\frac12\|u\|_{L^2}^2,\qquad F_\lambda:=(E_\lambda,J),
\]
and let $\operatorname{ind}_{-}(Q)$ denote the maximal dimension of a
subspace on which a quadratic form $Q$ is negative definite.  For every
$\kappa>0$ and integer $N\ge1$, there is $\Lambda>0$ with the following
property.  If $(\lambda,u)\in\mathcal A$ and
$(\bar p,\bar p_0)\in(\R^n)^*\times\R$ is nonzero and satisfies
\begin{equation}\label{eq:uniform-Goh-stationarity}
 \bar pD_uE_\lambda=\bar p_0D_uJ,\qquad
 |\bar p|\ge\Lambda|\bar p_0|,
\end{equation}
and
\begin{equation}\label{eq:uniform-Goh-index}
 \operatorname{ind}_{-}\!\left(
  \left.(\bar p,-\bar p_0)\cdot D_u^2F_\lambda
  \right|_{\ker D_uF_\lambda}\right)<N,
\end{equation}
then the adjoint covector $p(t)$ with terminal value $p(1)=\bar p$
satisfies
\begin{equation}\label{eq:uniform-approximate-Goh}
 \left|p(t)[X_i^\lambda,X_j^\lambda](\gamma_{0,u}^\lambda(t))\right|
 \le\kappa|\bar p|
 \qquad(0\le t\le1, 1\le i,j\le k).
\end{equation}
\end{lemma}
\begin{proof}
Set
\[
 R_*:=\sup\{\|u\|_{L^\infty}:(\lambda,u)\in\mathcal A\}
\]
and choose $L>R_*$.  Fix $(\lambda_0,\bar u)\in\mathcal A$.
For the fixed system $X^{\lambda_0}$, \cite[Proposition~3.7]{Rifford2023} yields, for every $\kappa > 0$ and $N \ge 1$, constants $\rho_{\lambda_0,\bar u} > 0$ and $\Lambda_{\lambda_0,\bar u} > 0$ such that the conclusion holds whenever
\[
    \| u - \bar u\|_{L^2} \le \rho_{\lambda_0,\bar u}, \qquad \| u\|_{L^\infty} < 2L, \qquad |\bar p| \ge \Lambda_{\lambda_0,\bar u}|\bar p_0|,
\]
together with the stationarity and index assumptions \eqref{eq:uniform-Goh-stationarity}--\eqref{eq:uniform-Goh-index}. Moreover, after possibly shrinking a neighborhood of $\lambda_0$, the radius $\rho_{\lambda_0,\bar u}$ may be chosen uniformly positive and the threshold $\Lambda_{\lambda_0,\bar u}$ uniformly bounded with respect to $\lambda$. Indeed, inspection of
\cite[Lemma~3.8 and equations~(3.19)--(3.32)]{Rifford2023} shows that
they depend only on $L$, on uniform bounds for the fields and their first
two derivatives, and on bounds for the fundamental matrix of the
linearized system and its inverse.  Joint smoothness and the common compact support provide the required $C^2$-bounds on the fields uniformly for $\lambda$ near $\lambda_0$. Moreover, if $S^{\lambda,u}$ denotes the fundamental matrix of the linearized system along $\gamma^{\lambda}_{0,u}$,
\[ \dot S^{\lambda,u}(t) = A^{\lambda,u} S^{\lambda,u}(t), \qquad A^{\lambda,u}(t) = \sum_{i=1}^k u_i(t) J_{X_i^\lambda} (\gamma_{0,u}^\lambda(t)),
\]
with $S^{\lambda,u}(0) = I$, then the bound $\| u\|_{L^\infty} < 2L$ and the uniform $C^1$-bounds on the fields imply, by Gronwall's inequality, uniform bounds on $S^{\lambda,u}$ and $(S^{\lambda,u})^{-1}.$
Thus there are a neighborhood
$V_{\lambda_0,\bar u}$ of $(\lambda_0,\bar u)$ in
$\Theta\times L^2$ and a constant
$\Lambda_{\lambda_0,\bar u}$ for which the conclusion holds throughout
$V_{\lambda_0,\bar u}\cap\mathcal A$.

The proof of the cited proposition also applies when $\bar u=0$: the
estimates in Lemma~3.8 and formulas (3.19)--(3.32) use only the
$L^\infty$ bound on the base control and the smoothness bounds just
described.

The sets $V_{\lambda_0,\bar u}\cap\mathcal A$ form an open cover of the
compact set $\mathcal A$ in its relative topology.  Choose a finite subcover and let $\Lambda$ be
the maximum of the corresponding thresholds.  Since
$\|u\|_{L^\infty}\le R_*<2L$ on $\mathcal A$, this $\Lambda$ has the
required property.
\end{proof}
\begin{lemma}[Uniform squared-distance stability]
\label{lem:app-distance-stability}
Let $C,C' \subset \mathbb{R}^n$ be compact sets such that
$0\in C\subset\operatorname{int}C'\Subset U_0$.  Suppose $Y^\lambda$ is
jointly smooth and uniformly two-generating on $\overline{U_1}$, and that
for some $R<\infty$ and every $\lambda\in\Theta$,
\begin{equation}\label{eq:centered-confinement-hypotheses}
 d_{Y^\lambda}(0,q)\le R\quad(q\in C'),
 \qquad
 \gamma_{0,u}^\lambda([0,1])\subset U_0
 \quad\text{if }\|u\|_{L^2}\le R+1.
\end{equation}
Then there is $L$ such that
\begin{equation}\label{eq:endpoint-lipschitz}
  |d_{Y^\lambda}(0,q)^2-d_{Y^\lambda}(0,q')^2|
  \le L|q-q'|
\end{equation}
for all $q,q'\in C$ and $\lambda\in\Theta$.
Moreover, if $Z^\lambda$ is a second family with the same properties and the same
bound $R$, then there are $\varepsilon_0,L'>0$ such that
\begin{equation}\label{eq:frame-distance-stability}
  \sup_{q\in C}
  |d_{Y^\lambda}(0,q)^2-d_{Z^\lambda}(0,q)^2|
  \le L'\max_i\|Y_i^\lambda-Z_i^\lambda\|_{C^1(\overline{U_1})}
\end{equation}
whenever the norm on the right is at most $\varepsilon_0$.
\end{lemma}
\begin{proof}
Choose compact sets $C_0,C_1,V,V'$ such that
\[
 C\subset\operatorname{int}C_0\subset C_0
 \subset\operatorname{int}C_1\subset C_1
 \subset\operatorname{int}C'
\]
and
\[
 C'\subset\operatorname{int}V
 \subset V\subset\operatorname{int}V'\Subset U_1.
\]
Lemma~\ref{lem:uniform-step-two-ball-box} gives constants $a\ge1$ and
$r_0>0$, independent of $\lambda$, such that
\begin{equation}\label{eq:appendix-local-accessibility}
 d_{Y^\lambda}(x,y)\le a|x-y|^{1/2}
\end{equation}
whenever $x,y\in V$ and $|x-y|<r_0$, with a realizing trajectory
contained in $V'$.

\medskip
\noindent \textbf{Step 1 (Compactness of endpoints and existence of minimizers).} We prove the confined, parametrized version of the standard endpoint
compactness argument, see e.g.~\cite[ Proposition~8.62]{ABB2020}. Suppose that
\[
 \lambda_m\longrightarrow\lambda,\qquad
 u_m\rightharpoonup u\quad\text{weakly in }L^2,\qquad
 \sup_m\|u_m\|_{L^2}\le R+1.
\]
The uniform $L^2$-bound on $(u_m)$ implies a uniform $(1/2)$-Hölder estimate on the curves $(\gamma_{0,u_m}^{\lambda_m})$. Therefore, by the Arzelà--Ascoli theorem, they converge uniformly to some $\tilde\gamma$, up to extraction of a subsequence. Joint smoothness of the fields and weak $L^2$ convergence allow passage to the limit in the integral equation defining the curves $(\gamma_{0,u_m}^{\lambda_m})$. Thus $\tilde \gamma$ solves the limiting controlled ODE with control $u$, and uniqueness yields $\tilde \gamma = \gamma_{0,u}^{\lambda}.$
Hence
\begin{equation}\label{eq:appendix-weak-endpoint-continuity}
 \operatorname{End}^{\lambda_m}_0(u_m)
 \longrightarrow\operatorname{End}^{\lambda}_0(u).
\end{equation}
Then the direct method in the calculus of variation yields a $Y^\lambda$-energy-minimizing control $u$ from $0$ to any $q \in C'$, with \begin{equation}\label{eq:appendix-minimizer-Linfty}
 |u(t)|=d_{Y^\lambda}(0,q)\le R
 \quad\text{for almost every }t.
\end{equation}
See e.g.~\cite[Corollary 8.64 and Lemma 3.64]{ABB2020}.

\medskip
\noindent { \textbf{Step 2 (Continuity and compactness of the minimizing family).}} We now establish the compactness of the set of confined, parametrized minimizers; see e.g.~\cite[Theorem 8.66]{ABB2020} for the classical version of this statement. Set
\[
 c(\lambda,q):=d_{Y^\lambda}(0,q)^2.
\]
We first show that the function $c$ is continuous on $\Theta\times C_1$. For a convergent sequence 
$(\lambda_m,q_m)\to(\lambda,q)$, let $u_m$ minimize
$c(\lambda_m,q_m)$. Then there exists $u \in L^2([0,1],\mathbb{R}^k)$ such that
$u_m\rightharpoonup u$ in $L^2$, up to extracting a subsequence. By
\eqref{eq:appendix-weak-endpoint-continuity}, $u$ joins $0$ to $q$ in the
$\lambda$-system, whence
\begin{equation}
 c(\lambda,q)\le\|u\|_{L^2}^2
 \le\liminf_{m\to\infty}c(\lambda_m,q_m).
\end{equation}
For the reverse inequality, let $v$ minimize $c(\lambda,q)$, and set $q'_m:=\operatorname{End}^{\lambda_m}_0(v).$ Continuous dependence on the parameter gives $q'_m\to q$.  For large $m$,
both $q'_m$ and $q_m$ belong to $V$, so that the triangle inequality and
\eqref{eq:appendix-local-accessibility} give
\[
\begin{split}
 d_{Y^{\lambda_m}}(0,q_m)
 &\le \|v\|_{L^2}
     +d_{Y^{\lambda_m}}(q'_m,q_m)\\
 &\le d_{Y^\lambda}(0,q)
     +a|q'_m-q_m|^{1/2}.
\end{split}
\]
Thus $\limsup_{m\to\infty}c(\lambda_m,q_m)\le c(\lambda,q)$, and continuity is proved. Let us show now that
\[
 \mathcal K:=
 \left\{(\lambda,u) \in \Theta\times L^2([0,1];\mathbb{R}^k)
 :
 \, \operatorname{End}^{\lambda}_0(u)\in C_1 \text{ and }
 \|u\|_{L^2}^2=c(\lambda,q)\right\}
\]
is compact in $\Theta\times L^2([0,1];\mathbb{R}^k)$, where $L^2([0,1];\mathbb{R}^k)$ is endowed with the strong topology. Note that
\eqref{eq:appendix-minimizer-Linfty} also yields a bound in $\Theta\times L^\infty([0,1];\mathbb{R}^k)$. After extraction, the parameters and endpoints of any sequence in
$\mathcal K$ converge and the controls converge weakly.  Endpoint
continuity from Step 1 and continuity of $c$ show that the weak limit is minimizing
and that the $L^2$ norms converge; hence the convergence is strong in $L^2$. 

\medskip
\noindent { \textbf{Step 3 (Reduction to compactly supported fields).}}
Choose $\chi\in C_c^\infty(U_1)$ equal to one near $\overline U_0$, set
$\widehat Y_i^\lambda:=\chi Y_i^\lambda$, and extend these fields by zero
to $\R^n$.  By a first-exit argument and \eqref{eq:centered-confinement-hypotheses}, every $\widehat Y^\lambda$-trajectory from $0$ with control norm at most $R + 1$ remains in $U_0$, where $\widehat Y^\lambda = Y^\lambda$. Since minimizing sequences for either distance eventually have norm below $R + 1$, and $d_{Y^\lambda}(0, q) \le R$, we obtain
\begin{equation}\label{eq:appendix-cutoff-distance}
 d_{\widehat Y^\lambda}(0,q)=d_{Y^\lambda}(0,q)\qquad(q\in C').)
\end{equation}  Thus $\mathcal K$ remains a compact family of minimizers for
the extended endpoint maps.  We now use the extended fields and again
denote them by $Y_i^\lambda$.

\medskip
\noindent {\textbf{Step 4 (Uniform bound on subgradients).}}
Set
\[
 E_\lambda(v):=\operatorname{End}_0^\lambda(v),\qquad
 J(v):=\frac12\|v\|_{L^2}^2,\qquad
 f_\lambda(q):=\frac12c(\lambda,q)\quad(q\in \operatorname{int}C_1).
\]
The argument used in Step 2 shows that $f_\lambda$ is continuous on $\operatorname{int}C_1$.
For a continuous function $f$ on an open subset of $\R^n$, denote by
$\partial^-f(q)$ its viscosity subdifferential: $p\in\partial^-f(q)$ if
$f$ has a $C^1$ support from below at $q$ with differential $p$.  The
proximal subdifferential $\partial^-_P f(q)$ is defined in the same way
with a $C^2$ support.

Let $\varphi$ be a $C^2$ lower support for $f_\lambda$ at $q$. For $v$ sufficiently close in $L^2$ to a minimizing control
$u$ from $0$ to $q$,
\[
J(v)\ge f_\lambda(E_\lambda(v))\ge \varphi(E_\lambda(v)),
\]
with equality at $v=u$. Thus $J-\varphi\circ E_\lambda$ has a local
minimum at $u$. Since $p=d\varphi(q)$, its first variation vanishes
and, on $\ker D_uE_\lambda$, its second variation is nonnegative:
\begin{equation}\label{eq:appendix-proximal-multiplier}
pD_uE_\lambda=D_uJ,\qquad D_u^2J(v)-pD_u^2E_\lambda(v)\ge0 \quad(v\in\ker D_uE_\lambda).
\end{equation}
This is the endpoint-map argument of
\cite[Proposition~3.2 and Remark~3.3]{Rifford2023}; the same
local-minimum argument applies when \(q=0\).

We now use these relations to bound $p$. Uniform two-generation gives
a constant $A$ such that
\begin{equation}\label{eq:appendix-dual-two-generation}
 |\xi|\le A\left(
   \max_i|\xi(Y_i^\lambda(x))|
   +\max_{i<j}|\xi([Y_i^\lambda,Y_j^\lambda](x))|
 \right)
\end{equation}
for $(\lambda,x)\in\Theta\times\overline U_0$ and every covector $\xi$. Indeed, this is the dual estimate associated with the uniformly bounded
right inverses in Definition~\ref{def:uniform-generating}; the second
maximum is understood as zero if its index set is empty.

Apply Lemma~\ref{lem:uniform-approximate-Goh} with
\[
 (\bar p,\bar p_0)=(-p,-1).
\]
The stationarity relation \eqref{eq:uniform-Goh-stationarity} follows from
\eqref{eq:appendix-proximal-multiplier}, and the quadratic form in
\eqref{eq:uniform-Goh-index} is
\[
 (-p,1)\cdot D_u^2(E_\lambda,J)
 =D_u^2J-pD_u^2E_\lambda.
\]
It is nonnegative on $\ker D_uE_\lambda$, hence has negative index zero
on $\ker D_u(E_\lambda,J)$.  If $|p|\ge\Lambda$, the adjoint covector
$\zeta$ with $\zeta(1)=-p$ therefore satisfies
\eqref{eq:uniform-approximate-Goh}.  The first relation in
\eqref{eq:appendix-proximal-multiplier}, written in adjoint form, also gives
\[
 \zeta(t)Y_i^\lambda(\gamma_{0,u}^\lambda(t))=-u_i(t)
 \quad\text{for almost every }t.
\]
The left side is continuous, so it defines a continuous representative of
$-u_i$; by \eqref{eq:appendix-minimizer-Linfty} its absolute value is at
most $R$ for every $t$.  At $t=1$, estimates
\eqref{eq:uniform-approximate-Goh} and
\eqref{eq:appendix-dual-two-generation} give
\[
 |p|\le A(R+\kappa|p|),
\]
and hence $|p|\le2AR$.  The case $|p|<\Lambda$ is immediate, so
\begin{equation}\label{eq:appendix-proximal-bound}
 |p|\le B:=\max\{\Lambda,2AR\}
 \quad \text{for any }p\in\partial^-_P f_\lambda(q),  \quad\text{for any }q\in \operatorname{int}C_1,
\end{equation}
uniformly in $\lambda$. By \cite[Proposition~2.2]{Rifford2023}, every viscosity subgradient is a
limit in $T^*\R^n$ of proximal subgradients at nearby points.  Hence
\begin{equation}\label{eq:appendix-viscosity-bound}
 |p|\le B
 \quad\text{for any }p\in\partial^- f_\lambda(q),  \quad\text{for any }q\in \operatorname{int}C_1,
\end{equation}
again uniformly in $\lambda$.

\medskip
\noindent {\textbf{Step 5 (Lipschitz continuity of the squared distance).}}
Rifford's viscosity-subgradient characterization of Lipschitz functions
\cite[Proposition~2.4]{Rifford2023} now shows that $f_\lambda$ is
$B$-Lipschitz on every Euclidean ball contained in $\operatorname{int}C_1$.  Set
\[
 r:=\frac12\operatorname{dist}
 \bigl(C_0,\R^n\setminus \operatorname{int}C_1\bigr)>0.
\]
If $q,q'\in C_0$ and $|q-q'|<r$, apply this characterization on
$B(q,2r)\subset \operatorname{int}C_1$ to obtain
\[
 |c(\lambda,q)-c(\lambda,q')|
 \le2B|q-q'|.
\]
For $|q-q'|\ge r$, use $0\le c(\lambda,\cdot)\le R^2$ on $C'$ to get
\[
 |c(\lambda,q)-c(\lambda,q')|
 \le\frac{2R^2}{r}|q-q'|.
\]
Thus there is $L_1$, independent of $\lambda$, such that
\begin{equation}\label{eq:appendix-Lipschitz-C0}
 |c(\lambda,q)-c(\lambda,q')|\le L_1|q-q'|
 \qquad(q,q'\in C_0).
\end{equation}
In particular, \eqref{eq:endpoint-lipschitz} holds with $L=L_1$. Note that the latter estimate actually holds for any $q,q' \in C_0$.

\medskip
\noindent { \textbf{Step 6 (Stability under perturbation of the frame).}}
Return to the original fields.  Let $Z^\lambda$ be a second family satisfying the same hypotheses as $Y^\lambda$, with the same bound $R$. Applying the preceding argument to both families and enlarging $L_1$ if necessary, we obtain a common $L_1$-Lipschitz bound for their squared-distance functions on $C_0$.  Put
\[
 \varepsilon:=
 \max_i\|Y_i^\lambda-Z_i^\lambda\|_{C^1(\overline U_1)}.
\]
Fix $q\in C$, let $u$ minimize the $Z^\lambda$-energy from $0$ to $q$,
and let $q_Y$ be the endpoint obtained by running the same control in the
$Y^\lambda$-system.  By the analogue of
\eqref{eq:appendix-minimizer-Linfty} for $Z^\lambda$,
$\|u\|_{L^\infty}\le R$.
The two trajectories remain in $U_0$; since they are driven by the same control, uniform $C^1$-bounds for the
fields and Gronwall's inequality give
\[
 |q_Y-q|\le A_R\varepsilon
\]
with $A_R$ independent of $\lambda$ and $q$.  Choose
$\varepsilon_0>0$ so that
\[
 A_R\varepsilon_0<
 \operatorname{dist}\bigl(C,\R^n\setminus\operatorname{int}C_0\bigr).
\]
Then $q_Y\in C_0$ when $\varepsilon\le\varepsilon_0$. In that case, as mentioned at the end of Step 5, we can apply \eqref{eq:endpoint-lipschitz} and get
\[
\begin{split}
 d_{Y^\lambda}(0,q)^2
 &\le d_{Y^\lambda}(0,q_Y)^2+L_1|q-q_Y|\\
 &\le\|u\|_{L^2}^2+L_1A_R\varepsilon\\
 &=d_{Z^\lambda}(0,q)^2+L_1A_R\varepsilon.
\end{split}
\]
Interchanging $Y^\lambda$ and $Z^\lambda$ proves
\eqref{eq:frame-distance-stability}, with $L'=L_1A_R$.
\end{proof}
We conclude with the localization lemma used in Section~\ref{sec:step-two}. It provides the uniform confinement and distance bounds required to apply the distance-stability estimate above to the rescaled and model frames. Let us briefly recall the context of this section:  $(M,\mathcal{D},g)$ is a smooth equiregular sub-Riemannian manifold of step $s \le 2$, the open set $O\subset M$ admits an adapted bracket frame $(X_1,\ldots,X_n)$ for which each $\Phi_x^{-1}$ is a privileged coordinate system. The associated rescaled fields $X^{x,h}:=(X_1^{x,h},\ldots,X_k^{x,h})$ are defined in
\eqref{eq:general-rescaled-fields}. They converge smoothly on compact sets to the model nilpotent fields $\widehat X := (\widehat X_{1,x},\ldots,\widehat X_{k,x})$.

\begin{lemma}[Uniform localization of parametrized frames]
\label{lem:uniform-confinement}
Let $K\Subset O$, and let
$C\subset\R^n$ be a compact coordinate box such that
$0\in\operatorname{int}C$.  There exist bounded open boxes
$U_0\Subset U_1$, with $C\subset U_0$, and a constant $R<\infty$ with the
following property.  Let $h_0>0$, and let
$W^{x,h}=(W_1^{x,h},\ldots,W_k^{x,h})$, $x\in K$ and $h \in [0,h_0]$,
be a jointly smooth family on a neighborhood of
$K\times[0,h_0]\times\overline{U_1}$ satisfying
\begin{equation}\label{eq:confinement-family-limit}
  W_i^{x,0}(\xi)=\widehat X_{i,x}(\xi),
  \qquad x\in K,\xi \in \overline{U_1},  \quad 1\le i\le k.
\end{equation}
Then there exists $h_*\in(0,h_0]$ such that:
\begin{enumerate}
\item The fields $W^{x,h}$ are
pointwise linearly independent and the family is uniformly two-generating
on $\overline{U_1}$ for any $h \in [0,h_*]$;
\item Every $W^{x,h}$-trajectory starting at
$0$ and controlled by $u \in L^2([0,1];\mathbb{R}^k)$ with $\|u\|_{L^2}\le R+1$ is defined on $[0,1]$
and remains in $U_0$;
\item For every  $x\in K$ and
$ h \in [0,h_*]$,
\begin{equation}\label{eq:centered-distance-bound}
  d_{W^{x,h},U_1}(0,q)\le R,
  \qquad q\in C,
\end{equation}
where $d_{W^{x,h},U_1}$ denotes the distance obtained by restricting
trajectories to $U_1$;
\item For
$W^{x,h}=X^{x,h}$, the distance in \eqref{eq:centered-distance-bound} agrees with $d_{x,h}(0,q)$ for $q\in C$, where
$d_{x,0}:=\widehat d_x$.
\end{enumerate}
\end{lemma}
\begin{proof}
The uniform tangent ball--box estimates and homogeneity imply that, for
every $L<\infty$, the sets $\widehat B_x(L)$, $x\in K$, are contained in a
common compact box.  Choose
\[
  A>\sup_{x\in K,\ q\in C}\widehat d_x(0,q),
\]
and set $R:=A+1$.  Choose $U_0\Subset U_1$ so that $C$ and the closures of
the tangent balls $\widehat B_x(R+1)$ lie in $U_0$ at positive distance from
$\partial U_0$.  For a family as in the statement, set
\[
  \rho(h):=\sup_{\substack{x\in K,\ 0\le t\le h\\1\le i\le k}}
  \|W_i^{x,t}-\widehat X_{i,x}\|_{C^1(\overline U_1)}.
\]
Then $\rho(h)\to0$ as $h\downarrow0$.  The tangent fields are pointwise
linearly independent and uniformly two-generating on
$K\times\overline U_1$.  These properties are open in the $C^1$ topology,
and the tangent bracket maps admit uniformly bounded right inverses.
Consequently, after decreasing $h_0$, the same conclusions hold for
$W^{x,h}$.
Let $\gamma$ and $\widehat\gamma$ be the $W^{x,h}$- and
$\widehat X_x$-trajectories starting at $0$ with the same control
$\|u\|_{L^2}\le R+1$.  Every prefix control $u$ of $\widehat\gamma$ has length at most
$\|u\|_{L^1}\le R+1$, so $\widehat\gamma$ remains in
$\widehat B_x(R+1)$.  Up to the first exit of $\gamma$ from $U_0$,
Gronwall's inequality gives
\[
  \sup_t|\gamma(t)-\widehat\gamma(t)|\le c_R\rho(h).
\]
For sufficiently small $h$, the right-hand side is smaller than the distance
of the tangent balls from $\partial U_0$.  Therefore  $\gamma$ is defined on $[0,1]$ and remains in $U_0$.
We next prove \eqref{eq:centered-distance-bound}.   Choose compact sets $V,V'$ such that
\[
  C\subset\operatorname{int}V\subset V
  \subset\operatorname{int}V'\Subset U_0.
\]
Fix $q\in C$, choose an
$\widehat X_x$-control of $L^2$-norm less than $A$ that reaches $q$, and use the
same control for $W^{x,h}$.  Denote the resulting endpoint by $q_h$.  Then
$|q_h-q|=O(\rho(h))$, and $q,q_h\in V$ for small $h$.  By Lemma~\ref{lem:uniform-step-two-ball-box}, the points $q_h$
and $q$ can be joined inside $V'$ at cost $O(\rho(h)^{1/2})$.  Concatenating
the two controls and using constant-speed parametrization gives a control
of norm
\[
  A+O(\rho(h)^{1/2})<R.
\]
This proves \eqref{eq:centered-distance-bound}. It remains to identify the distance for $W^{x,h}=X^{x,h}$.  It is obvious that
\[
  d_{x,h}(0,q)\le d_{X^{x,h},U_1}(0,q),
  \qquad q\in C.
\]
Conversely, uniform convergence
\[
d_{x,h}(0,\cdot)\longrightarrow \widehat d_x(0,\cdot)
\]
on \(K\times C\) gives \(d_{x,h}(0,q)<R\) for all sufficiently small \(h\). Fix
\[
0<\varepsilon<R+1-d_{x,h}(0,q).
\]
By the definition of \(d_{x,h}\), there exists a constant-speed ambient horizontal curve \(\gamma:[0,1]\to M\) such that
\[
\gamma(0)=x,\qquad
\gamma(1)=\Phi_x(\delta_hq),
\qquad
L(\gamma)<h\bigl(d_{x,h}(0,q)+\varepsilon\bigr).
\]
Let \(\tau\) be the first exit time of \(\gamma\) from
\(\Phi_x(\delta_hU_0)\), with \(\tau=1\) if no exit occurs. For
\(0\le t<\tau\), write
\[
\dot\gamma(t)=\sum_{i=1}^k v_i(t)X_i(\gamma(t)),
\]
and set \(u_i:=h^{-1}v_i\). If \(\tau<1\), extend \(u\) by zero on
\([\tau,1]\). Since \(\gamma\) has constant speed,
\[
\|u\|_{L^2([0,1];\mathbb R^k)}
\le \frac{L(\gamma)}h
<d_{x,h}(0,q)+\varepsilon<R+1.
\]
For \(0\le t<\tau\), the rescaled curve
\[
\eta(t):=\delta_{1/h}\Phi_x^{-1}(\gamma(t))
\]
satisfies
\[
\dot\eta(t)
=\sum_{i=1}^k u_i(t)X_i^{x,h}(\eta(t)),
\qquad
\eta(0)=0.
\]
Let \(\widetilde\eta\) be the \(X^{x,h}\)-trajectory starting at \(0\) with control \(u\). By the confinement conclusion,
\[
\widetilde\eta([0,1])\subset U_0.
\]
Uniqueness gives \(\eta=\widetilde\eta\) on \([0,\tau)\). Hence \(\tau<1\) would imply that \(\widetilde\eta(\tau)\in\partial U_0\), contradicting the preceding inclusion. Therefore \(\tau=1\), and
\[
\eta([0,1])\subset U_0\Subset U_1,
\qquad
\eta(1)=q.
\]
Thus \(\eta\) is admissible for the \(U_1\)-restricted distance, and
\[
d_{X^{x,h},U_1}(0,q)
\le \|u\|_{L^2}
<d_{x,h}(0,q)+\varepsilon.
\]
Letting \(\varepsilon\downarrow0\) yields
\[
d_{X^{x,h},U_1}(0,q)\le d_{x,h}(0,q),
\]
which is the reverse inequality. For $h=0$, the same localization argument gives
$d_{\widehat X_x,U_1}(0,q)=\widehat d_x(0,q)$.
\end{proof}

\vspace{1cm}
\subsection*{Acknowledgments}
F.B.\ is partially supported by grant 10.46540/4283-00175B from the
Independent Research Fund Denmark, by the Villum Investigator grant
\emph{Stochastic Analysis in Aarhus}, and by the European Research Council
(ERC) under the European Union's Horizon Europe research and innovation
programme (RanGe project, Grant Agreement No.~101199772). J.J.\ is partially supported by grant OCENW.M20.251 from the research program Open Competitie
ENW, which is (partly) financed by the Dutch Research Council (NWO). D.T.\ is
supported by the Research Foundation--Flanders (FWO), Odysseus~II programme,
grant no.~G0DBZ23N.

\subsection*{Declaration of generative AI and AI-assisted technologies in the
writing process}
The authors are responsible for the architecture of the paper  and  proofs of the main results. During the preparation of this manuscript, the authors used ChatGPT
(Sol 5.6) to assist with editorial restructuring, proof presentation, bibliographical research, symbolic and numeric computations and
\LaTeX{} consistency. The authors   reviewed and edited all suggestions and
take full responsibility for the content of the article.

\end{document}